\documentclass[12pt]{amsart}

\usepackage{accents}
\usepackage{appendix}
\usepackage{amsfonts}
\usepackage{amsmath}
\usepackage{amssymb}	
\usepackage{amsthm,bm}
\usepackage{array,booktabs,multirow}
\usepackage{braket}
\usepackage{centernot}
\usepackage{cite}
\usepackage{comment}
\usepackage{dsfont}
\usepackage{mathalfa}
\usepackage[shortlabels]{enumitem} 
\usepackage{etoolbox}
\usepackage{float}
\usepackage[hang, flushmargin]{footmisc}
\usepackage{latexsym}
\usepackage{lipsum}
\usepackage{needspace}
\usepackage{tikz}
\usepackage{hyperref}
\usetikzlibrary{matrix,arrows}  
\usepackage{kbordermatrix}

\renewcommand{\kbldelim}{(}
\renewcommand{\kbrdelim}{)}
\newcommand{\rvline}{\hspace*{-\arraycolsep}\vline\hspace*{-\arraycolsep}}

\theoremstyle{plain}
\newtheorem{thm}{Theorem}[section]

\newtheorem{lem}[thm]{Lemma}
\newtheorem{prop}[thm]{Proposition}

\theoremstyle{definition}

\newtheorem{defn}[thm]{Definition}

\theoremstyle{remark}

\setlist[enumerate,1]{leftmargin=2em}

\def\H{\mathfrak{H}}

\def\F{\mathbb F}
\def\Z{\mathbb Z}

\def\e{\varepsilon}

\title[Finite-dimensional irreducible modules of DAHA of type $(C_1^\vee,C_1)$]{Finite-dimensional irreducible modules of the \\
universal DAHA of  type $(C_1^\vee,C_1)$}

\author{Hau-Wen Huang}
\address{
Hau-Wen Huang\\
Department of Mathematics\\
National Central University\\
Chung-Li 32001 Taiwan
}
\email{hauwenh@math.ncu.edu.tw}

\begin{document}
\begin{abstract}
Assume that $\F$ is an algebraically closed field and let $q$ denote a nonzero scalar in $\F$ that is not a root of unity. The universal DAHA (double affine Hecke algebra) $\H_q$ of type $(C_1^\vee,C_1)$ is a unital associative $\F$-algebra defined by generators and relations. The generators are $\{t_i^{\pm 1}\}_{i=0}^3$ and the relations assert that 
\begin{gather*}
t_it_i^{-1}=t_i^{-1} t_i=1
\quad 
\hbox{for all $i=0,1,2,3$};
\\
\hbox{$t_i+t_i^{-1}$ is central} 
\quad 
\hbox{for all $i=0,1,2,3$};
\\
t_0t_1t_2t_3=q^{-1}.
\end{gather*}
In this paper we describe the finite-dimensional irreducible $\H_q$-modules from many viewpoints and classify the finite-dimensional irreducible $\H_q$-modules up to isomorphism. The proofs are carried out in the language of linear algebra. 
\end{abstract}

\maketitle

{\footnotesize{\bf Keywords:} double affine Hecke algebras, irreducible modules, universal property.}

{\footnotesize{\bf MSC2020:} 20C08, 33D45, 33D80.}

\allowdisplaybreaks

\section{Introduction}
In the nineties, the double affine Hecke algebra (DAHA) was introduced by Cherednik in connection with quantum affine Knizhnik--Zamolodchikov equations and Macdonald eigenvalue problems \cite{DAHA_book,Cherednik:1992,Cherednik:1995}. Since that time, DAHAs and their representations have been explored in many other areas such as algebraic geometry \cite{DAHA:2003,DAHA:2005,DAHA:2009,Oblomkov2009}, combinatorics \cite{Tanaka:2018,Lee:2017}, 
integrable systems \cite{DAHA2007,DAHA&MAT2:2016,nilDAHA&qHahn:2016,nilDAHA&qHahn:2018} 
and 
orthogonal polynomials 
\cite{DAHA2013,DAHA:2008,Huang:R<BImodules,
DAHA&OP_book,koo07,koo08,daha&Z3,daha&LP}.

More than 30 years ago, a classification of irreducible modules of  affine Hecke algebras was already given in \cite{AHA87}. Despite this, the classification of finite-dimensional irreducible modules of DAHAs is still open up to now. We are here concerned with the representation theory of DAHA of type $(C_1^\vee,C_1)$.  Throughout this paper, let $\F$ denote a field and let $q$ be a nonzero scalar in $\F$. The DAHA of type $(C_1^\vee,C_1)$ is defined as follows:

\begin{defn}
Given nonzero scalars $k_0,k_1,k_2,k_3\in \F$ the {\it DAHA $\H_q(k_0,k_1,k_2,k_3)$ of type $(C_1^\vee,C_1)$} is a unital associative $\F$-algebra generated by $t_0,t_1,t_2,t_3$ subject to the following relations:
\begin{gather*}
(t_i-k_i)(t_i-k_i^{-1})=0 
\quad \hbox{for all $i=0,1,2,3$};
\\
t_0 t_1 t_2 t_3 =q^{-1}.
\end{gather*}
\end{defn}

\noindent We mainly restrict our attention to the case when $\F$ is algebraically closed and $q$ is not a root of unity.
Relying on the Crawley-Boevey's results on the Deligne--Simpson problem \cite{CrawleyBoevey2004}, the finite-dimensional irreducible $\H_q(k_0,k_1,k_2,k_3)$-modules was first classified in \cite{Oblomkov2009}. In this paper, we consider a central extension of DAHA of type $(C_1^\vee,C_1)$ defined as follows:

\begin{defn}
[Definition 3.1, \cite{DAHA2013}]
\label{defn:DAHA}
The {\it universal DAHA $\H_q$ of type $(C_1^\vee,C_1)$} is a unital associative $\F$-algebra defined by generators and relations. The generators are $\{t_i^{\pm 1}\}_{i=0}^3$ and the relations assert that 
\begin{gather}
t_it_i^{-1}=t_i^{-1} t_i=1
\quad 
\hbox{for all $i=0,1,2,3$};
\\
\hbox{$t_i+t_i^{-1}$ is central} 
\quad 
\hbox{for all $i=0,1,2,3$};
\\
t_0t_1t_2t_3=q^{-1}.
\label{t0t1t2t3}
\end{gather}
\end{defn}

In this paper, we describe the finite-dimensional irreducible $\H_q$-modules from many viewpoints and classify the finite-dimensional irreducible $\H_q$-modules without the aid of the Crawley-Boevey's results. Our proofs mainly use elementary linear algebra techniques. The outline of this paper is as follows:  
Given appropriate parameters $k_0,k_1,k_2,k_3$ we construct an even-dimensional $\H_q$-module $E(k_0,k_1,k_2,k_3)$ and describe how $\{t_i\}_{i=0}^3$ act on a basis of $E(k_0,k_1,k_2,k_3)$. 
In Theorem \ref{thm:iso}, we claim that if $E(k_0,k_1,k_2,k_3)$ is irreducible, then it is isomorphic to $E(k_0,k_1^{\pm 1},k_2^{\pm 1},k_3^{\pm 1})$. In Theorem \ref{thm:even}, we state a classification of even-dimensional $\H_q$-modules via the $\H_q$-modules $E(k_0,k_1,k_2,k_3)$. In Theorems \ref{thm:iso_O} and \ref{thm:odd} we state similar results on odd-dimensional $\H_q$-modules. See \S\ref{s:result} for details. In \S\ref{s:Verma} we display an infinite-dimensional $\H_q$-module $M(k_0,k_1,k_2,k_3)$ and establish its universal property. Motivated by the universal property, we discuss the existence of the simultaneous eigenvectors of $t_0$ and $t_3$ in finite-dimensional irreducible $\H_q$-modules in \S\ref{s:eigenvector}. Subsequently we show that $E(k_0,k_1,k_2,k_3)$ is a quotient of $M(k_0,k_1,k_2,k_3)$. We give a necessary and sufficient condition for $E(k_0,k_1,k_2,k_3)$ as irreducible in terms of the parameters $k_0,k_1,k_2,k_3,q$. Moreover we give a proof for Theorem \ref{thm:iso}. See \S\ref{section:iso} for details. In \S\ref{section:even} we show that any even-dimensional irreducible $\H_q$-module is obtained by twisting the $\H_q$-module $E(k_0,k_1,k_2,k_3)$ up to isomorphism. This leads to a proof for Theorem \ref{thm:even}. In \S\ref{section:iso_O} and \S\ref{section:odd} we prove Theorems \ref{thm:iso_O} and \ref{thm:odd}, respectively.

We have some remarks on the related literature. Inspired by \cite{koo07,koo08}, it was shown that the universal Askey--Wilson algebra $\triangle_q$ is an $\F$-subalgebra of $\H_q$ \cite{DAHA2013}. In \cite{Huang:2015}, it was given a classification of finite-dimensional irreducible $\triangle_q$-modules. In \cite{Huang:AW&DAHAmodule}, the present author classifies the lattices of $\triangle_q$-submodules of finite-dimensional irreducible $\H_q$-modules.

\section{Statement of results}\label{s:result}

In this section we state our main results (Theorems \ref{thm:iso}, \ref{thm:even}, \ref{thm:iso_O} and \ref{thm:odd}). In Propositions \ref{prop:E} and \ref{prop:O}, given appropriate parameters $k_0,k_1,k_2,k_3$ we construct an even-dimensional $\H_q$-module $E(k_0,k_1,k_2,k_3)$ and an odd-dimensional $\H_q$-module $O(k_0,k_1,k_2,k_3)$.
In Theorems \ref{thm:iso} and \ref{thm:iso_O} we display several equivalent descriptions of the finite-dimensional $\H_q$-modules $E(k_0,k_1,k_2,k_3)$ and $O(k_0,k_1,k_2,k_3)$ when they are irreducible.  In Theorems \ref{thm:even}  and \ref{thm:odd} we give our classification of finite-dimensional irreducible $\H_q$-modules via the $\H_q$-modules $E(k_0,k_1,k_2,k_3)$ and $O(k_0,k_1,k_2,k_3)$, provided that $\F$ is algebraically closed and $q$ is not a root of unity. In \S\ref{section:iso}--\S\ref{section:odd} we will give our proofs for the main  results.

Define 
\begin{align}\label{ci}
c_i&=t_i+t_i^{-1}
\qquad 
\hbox{for all $i=0,1,2,3$}.
\end{align}
Recall from Definition \ref{defn:DAHA} that $c_i$ is central in $\H_q$. 



\begin{prop}\label{prop:E}
Let $d\geq 1$ denote an odd integer. Assume that $k_0,k_1,k_2,k_3$ are nonzero scalars in $\F$ with 
$$
k_0^2=q^{-d-1}.
$$ 
Then there exists a $(d+1)$-dimensional $\H_q$-module $E(k_0,k_1,k_2,k_3)$ satisfying the following conditions:
\begin{enumerate}
\item There exists an $\F$-basis $\{v_i\}_{i=0}^d$ for $E(k_0,k_1,k_2,k_3)$ such that 
\begin{align*}
t_0 v_i 
&=
\left\{
\begin{array}{ll}
\textstyle
k_0^{-1} q^{-i} (1-q^i) (1-k_0^2 q^i) 
v_{i-1}
+
(
k_0+k_0^{-1}-k_0^{-1}q^{-i}
) 
v_i
\qquad
&\hbox{for $i=2,4,\ldots,d-1$},
\\
k_0^{-1} q^{-i-1}
(v_i-v_{i+1})
\qquad
&\hbox{for $i=1,3,\ldots,d-2$},
\end{array}
\right.
\\
t_0 v_0&=k_0 v_0,
\qquad
t_0 v_d=k_0v_d,
\\
t_1 v_i
&=
\left\{
\begin{array}{ll}
-k_1(1-q^i)(1-k_0^2 q^i)v_{i-1}
+k_1 v_i
+k_1^{-1} v_{i+1}
\qquad
&\hbox{for $i=2,4,\ldots,d-1$},
\\
k_1^{-1} v_i
\qquad
&\hbox{for $i=1,3,\ldots,d$},
\end{array}
\right.
\\
t_1 v_0 &=k_1 v_0+k_1^{-1} v_1,
\\
t_2 v_i 
&=
\left\{
\begin{array}{ll}
k_0^{-1} k_1^{-1} k_3^{-1} q^{-i-1}
(v_i-v_{i+1})
\qquad
&\hbox{for $i=0,2,\ldots,d-1$},
\\
\textstyle
\frac{(k_0 k_1 k_3 q^i-k_2)
(k_0 k_1 k_3 q^i- k_2^{-1})}
{k_0 k_1 k_3 q^i  }
v_{i-1}
+
(k_2+k_2^{-1}-
k_0^{-1} k_1^{-1} k_3^{-1} q^{-i}
) v_i
\qquad
&\hbox{for $i=1,3,\ldots,d$},
\end{array}
\right.
\\
t_3 v_i 
&=
\left\{
\begin{array}{ll}
k_3 v_i
\qquad
&\hbox{for $i=0,2,\ldots,d-1$},
\\
-k_3^{-1}(k_0k_1k_3q^i-k_2)
(k_0k_1k_3 q^i-k_2^{-1})
v_{i-1}
+k_3^{-1} v_i
+k_3 v_{i+1}
\qquad
&\hbox{for $i=1,3,\ldots,d-2$}.
\end{array}
\right.
\\
t_3 v_d &=
-k_3^{-1}(k_0k_1k_3q^d-k_2)
(k_0k_1k_3 q^d-k_2^{-1})
v_{d-1}
+k_3^{-1} v_d.
\end{align*}

\item The elements $c_0,c_1,c_2,c_3$ act on $E(k_0,k_1,k_2,k_3)$ as scalar multiplication by $$
k_0+k_0^{-1},
\quad 
k_1+k_1^{-1},
\quad 
k_2+k_2^{-1},
\quad 
k_3+k_3^{-1}
$$ 
respectively.
\end{enumerate}
\end{prop}
\begin{proof}
It is routine to verify the proposition by using Definition \ref{defn:DAHA}.
\end{proof}

The proof of the following result concerning the isomorphism class of the $\H_q$-module $E(k_0,k_1,k_2,k_3)$ is given in \S\ref{section:iso}.

\begin{thm}\label{thm:iso}
Let $d\geq 1$ denote an odd integer. Assume that $k_0,k_1,k_2,k_3$ are nonzero scalars in $\F$ with 
$
k_0^2=q^{-d-1}$.  
If the $\H_q$-module $E(k_0,k_1,k_2,k_3)$ is irreducible then 
the following hold: 
\begin{enumerate}
\item $E(k_0,k_1,k_2,k_3)$ is isomorphic to the $\H_q$-module $E(k_0,k_1^{-1},k_2,k_3)$.
\item $E(k_0,k_1,k_2,k_3)$ is isomorphic to the $\H_q$-module  $E(k_0,k_1,k_2^{-1},k_3)$.
\item $E(k_0,k_1,k_2,k_3)$ is isomorphic to the $\H_q$-module $E(k_0,k_1,k_2,k_3^{-1})$.
\end{enumerate}
\end{thm}

Let $V$ denote an $\H_q$-module. For any $\F$-algebra automorphism $\e$ of $\H_q$ the notation  
$$
V^\e
$$
stands for the $\H_q$-module obtained by twisting the $\H_q$-module $V$ via $\e$.

 Let $\Z$ denote the additive group of integers. Recall that $\Z/4\Z$ is the additive group of integers modulo $4$. 
Observe that there exists a unique $\Z/4\Z$-action on $\H_q$ such that each element of $\Z/4\Z$ acts on $\H_q$ as an $\F$-algebra automorphism in the following way:

\begin{table}[H]
\centering
\extrarowheight=3pt
\begin{tabular}{c|rrrr}
$\e\in \Z/4\Z$  &$t_0$ &$t_1$ &$t_2$ &$t_3$ 
\\

\midrule[1pt]

${0\pmod 4}$ &$t_0$  &$t_1$ &$t_2$ &$t_3$ 
\\
${1\pmod 4}$ &$t_1$ &$t_2$ &$t_3$ &$t_0$ 
\\
${2\pmod 4}$ &$t_2$ &$t_3$ &$t_0$ &$t_1$
\\
${3\pmod 4}$ &$t_3$ &$t_0$ &$t_1$ &$t_2$
\end{tabular}
\caption{The $\Z/4\Z$-action on $\H_q$}\label{Z/4Z-action}
\end{table}

The classification of even-dimensional irreducible $\H_q$-modules is stated as follows and the proof is given in \S\ref{section:even}.

\begin{thm}\label{thm:even}
Assume that $\F$ is algebraically closed and $q$ is not a root of unity. 
Let $d\geq 1$ denote an odd integer. 
Let $\mathbf{EM}_d$ denote the set of all isomorphism classes of irreducible $\H_q$-modules that have dimension $d+1$. Let $\mathbf{EP}_d$ denote the set of all quadruples $(k_0,k_1,k_2,k_3)$ of nonzero scalars in $\F$ that satisfy $
k_0^2=q^{-d-1}$ and 
\begin{gather*}
k_0k_1k_2k_3, k_0k_1^{-1}k_2k_3, k_0k_1k_2^{-1}k_3,  k_0k_1k_2k_3^{-1}
\not=
q^{-i}
\qquad
\hbox{for all $i=1,3,\ldots,d$}.
\end{gather*}
Define an action of the abelian group $\{\pm 1\}^3$ on $\mathbf{EP}_d$ by 
\begin{align*}
(k_0,k_1,k_2,k_3)^{(-1,1,1)} &= (k_0,k_1^{-1},k_2,k_3),
\\
(k_0,k_1,k_2,k_3)^{(1,-1,1)} &= (k_0,k_1,k_2^{-1},k_3),
\\
(k_0,k_1,k_2,k_3)^{(1,1,-1)} &= (k_0,k_1,k_2,k_3^{-1})
\end{align*}
for all $(k_0,k_1,k_2,k_3)\in \mathbf{EP}_d$. Let $\mathbf{EP}_d/\{\pm 1\}^3$ denote the set of the $\{\pm 1\}^3$-orbits of $\mathbf{EP}_d$. For $(k_0,k_1,k_2,k_3)\in \mathbf{EP}_d$ let $[k_0,k_1,k_2,k_3]$ denote the $\{\pm 1\}^3$-orbit of $\mathbf{EP}_d$ that contains $(k_0,k_1,k_2,k_3)$. Then there exists a bijection $\mathcal E:\Z/4\Z \times \mathbf{EP}_d/\{\pm 1\}^3 \to \mathbf{EM}_d$ given by
\begin{eqnarray*}
(\e,[k_0,k_1,k_2,k_3])
&\mapsto &
\hbox{the isomorphism class of $E(k_0,k_1,k_2,k_3)^\e$}
\end{eqnarray*}
for all $\e\in \Z/4\Z$ and all $[k_0,k_1,k_2,k_3]\in \mathbf{EP}_d/\{\pm 1\}^3$.
\end{thm}

An odd-dimensional $\H_q$-module $O(k_0,k_1,k_2,k_3)$ with four  parameters $k_0,k_1,k_2,k_3$ is built as follows:

\begin{prop}\label{prop:O}
Let $d\geq 0$ denote an even integer. Assume that $k_0,k_1,k_2,k_3$ are nonzero scalars in $\F$ with 
$$
k_0 k_1 k_2 k_3=q^{-d-1}.
$$
Then there exists a $(d+1)$-dimensional $\H_q$-module
 $O(k_0,k_1,k_2,k_3)$ satisfying the following conditions:
\begin{enumerate} 
\item There exists an $\F$-basis $\{v_i\}_{i=0}^d$ for $O(k_0,k_1,k_2,k_3)$  such that 
\begin{align*}
t_0 v_i 
&=
\left\{
\begin{array}{ll}
\textstyle
k_0^{-1} q^{-i} (1-q^i) (1-k_0^2 q^i) 
v_{i-1}
+
(
k_0+k_0^{-1}-k_0^{-1}q^{-i}
) 
v_i
\qquad
&\hbox{for $i=2,4,\ldots,d$},
\\
\textstyle
k_0^{-1} q^{-i-1}
(v_i-v_{i+1})
\qquad
&\hbox{for $i=1,3,\ldots,d-1$},
\end{array}
\right.
\\
t_0 v_0&= k_0 v_0,
\\
t_1 v_i
&=
\left\{
\begin{array}{ll}
-k_1(1-q^i)(1-k_0^2 q^i)v_{i-1}
+k_1 v_i
+k_1^{-1} v_{i+1}
\qquad
&\hbox{for $i=2,4,\ldots,d-2$},
\\
k_1^{-1} v_i
\qquad
&\hbox{for $i=1,3,\ldots,d-1$},
\end{array}
\right.
\\
t_1 v_0&= k_1 v_0 +k_1^{-1} v_1,
\qquad 
t_1 v_d=-k_1(1-q^d)(1-k_0^2 q^d)v_{d-1}
+k_1 v_d,
\\
t_2 v_i 
&=
\left\{
\begin{array}{ll}
k_2 q^{d-i}
(v_i-v_{i+1})
\qquad
&\hbox{for $i=0,2,\ldots,d-2$},
\\
\textstyle
-k_2(1-k_2^{-2}q^{i-d-1})
(1-q^{d-i+1})
v_{i-1}
+
(k_2+k_2^{-1}-
k_2 q^{d-i+1}) v_i
\qquad
&\hbox{for $i=1,3,\ldots,d-1$},
\end{array}
\right.
\\
t_2 v_d &= k_2 v_d,
\\
t_3 v_i 
&=
\left\{
\begin{array}{ll}
k_3 v_i
\qquad
&\hbox{for $i=0,2,\ldots,d$},
\\
\textstyle
-k_3^{-1}
(1-k_2^{-2} q^{i-d-1})
(1-q^{i-d-1})
v_{i-1}
+k_3^{-1} v_i
+k_3 v_{i+1}
\qquad
&\hbox{for $i=1,3,\ldots,d-1$}.
\end{array}
\right.
\end{align*}
\item The elements $c_0,c_1,c_2,c_3$ act on $O(k_0,k_1,k_2,k_3)$ as scalar multiplication by 
$$
k_0+k_0^{-1},
\quad 
k_1+k_1^{-1},
\quad 
k_2+k_2^{-1},
\quad 
k_3+k_3^{-1}
$$ respectively.
\end{enumerate}
\end{prop}
\begin{proof}
It is routine to verify the proposition by using Definition \ref{defn:DAHA}.
\end{proof}

The proof of the following result concerning the isomorphism class of the $\H_q$-module $O(k_0,k_1,k_2,k_3)$ is given in \S\ref{section:iso_O}.

\begin{thm}\label{thm:iso_O}
Let $d\geq 0$ denote an even integer. Assume that $k_0,k_1,k_2,k_3$ are nonzero scalars in $\F$ with 
$
k_0 k_1 k_2 k_3=q^{-d-1}$. 
If the $\H_q$-module $O(k_0,k_1,k_2,k_3)$ is irreducible then the following hold:
\begin{enumerate}
\item $O(k_0,k_1,k_2,k_3)$ is isomorphic to the $\H_q$-module $O(k_1,k_2,k_3,k_0)^{3\bmod {4}}$.

\item $O(k_0,k_1,k_2,k_3)$ is isomorphic to the $\H_q$-module $O(k_2,k_3,k_0,k_1)^{2\bmod {4}}$.

\item $O(k_0,k_1,k_2,k_3)$ is isomorphic to the $\H_q$-module  $O(k_3,k_0,k_1,k_2)^{1\bmod {4}}$.
\end{enumerate}
\end{thm}

The classification of odd-dimensional irreducible $\H_q$-modules is stated as follows and the proof is given in \S\ref{section:odd}.

\begin{thm}\label{thm:odd}
Assume that $\F$ is algebraically closed and $q$ is not a root of unity. 
Let $d\geq 0$ denote an even integer. 
Let $\mathbf{OM}_d$ denote the set of all isomorphism classes of irreducible $\H_q$-modules that have dimension $d+1$. Let $\mathbf{OP}_d$ denote the set of all quadruples $(k_0,k_1,k_2,k_3)$ of nonzero scalars in $\F$ that satisfy $k_0 k_1 k_2 k_3=q^{-d-1}$ and 
$$
k_0^2,k_1^2,k_2^2,k_3^2\not=
q^{-i}
\qquad 
\hbox{for all $i=2,4,\ldots,d$}.
$$
Then there exists a bijection $\mathcal O:\mathbf{OP}_d\to \mathbf{OM}_d$ given by
\begin{eqnarray*}
(k_0,k_1,k_2,k_3)
&\mapsto &
\hbox{the isomorphism class of $O(k_0,k_1,k_2,k_3)$}
\end{eqnarray*}
for all $(k_0,k_1,k_2,k_3)\in \mathbf{OP}_d$.
\end{thm}


\section{An infinite-dimensional $\H_q$-module and its universal property}\label{s:Verma}

In this section, we present an infinite-dimensional $\H_q$-module and its universal property.

\begin{prop}\label{prop:P}
For any nonzero scalars $k_0,k_1,k_2,k_3\in \F$
there exists an $\H_q$-module $M(k_0,k_1,k_2,k_3)$ satisfying the following conditions:
\begin{enumerate}
\item There exists  
an $\F$-basis $\{m_i\}_{i=0}^\infty$ for  $M(k_0,k_1,k_2,k_3)$ such that 
\begin{align*}
t_0 m_i 
&=
\left\{
\begin{array}{ll}
\textstyle
k_0^{-1} q^{-i} (1-q^i) (1-k_0^2 q^i) 
m_{i-1}
+
(
k_0+k_0^{-1}-k_0^{-1}q^{-i}
) 
m_i
\qquad
&\hbox{if $i=2,4,\ldots$},
\\
k_0^{-1} q^{-i-1}
(m_i-m_{i+1})
\qquad
&\hbox{if $i=1,3,\ldots$},
\end{array}
\right.
\\
t_0 m_0 &= k_0 m_0,
\\
t_1 m_i
&=
\left\{
\begin{array}{ll}
-k_1(1-q^i)(1-k_0^2 q^i)m_{i-1}
+k_1 m_i
+k_1^{-1} m_{i+1}
\qquad
&\hbox{if $i=2,4,\ldots$},
\\
k_1^{-1} m_i
\qquad
&\hbox{if $i=1,3,\ldots$},
\end{array}
\right.
\\
t_1 m_0 &=k_1 m_0+k_1^{-1} m_1,
\\
t_2 m_i 
&=
\left\{
\begin{array}{ll}
k_0^{-1} k_1^{-1} k_3^{-1} q^{-i-1}
(m_i-m_{i+1})
\qquad
&\hbox{if $i=0,2,\ldots$},
\\
\frac{(k_0 k_1 k_3 q^i-k_2)(k_0 k_1 k_3 q^i-k_2^{-1})}{k_0 k_1 k_3 q^i}
m_{i-1}
+
(k_2+k_2^{-1}-
k_0^{-1} k_1^{-1} k_3^{-1} q^{-i}
) m_i
\qquad
&\hbox{if $i=1,3,\ldots$},
\end{array}
\right.
\\
t_3 m_i 
&=
\left\{
\begin{array}{ll}
k_3 m_i
\qquad
&\hbox{if $i=0,2,\ldots$},
\\
-k_3^{-1}(k_0k_1k_3q^i-k_2)
(k_0k_1k_3 q^i-k_2^{-1} )
m_{i-1}
+k_3^{-1} m_i
+k_3 m_{i+1}
\qquad
&\hbox{if $i=1,3,\ldots$}.
\end{array}
\right.
\end{align*}
\item The elements $c_0,c_1,c_2,c_3$ act on $M(k_0,k_1,k_2,k_3)$ as scalar multiplication by 
$$
k_0+k_0^{-1},
\qquad 
k_1+k_1^{-1},
\qquad 
k_2+k_2^{-1},
\qquad 
k_3+k_3^{-1}
$$ respectively.
\end{enumerate}
\end{prop}
\begin{proof}
It is routine to verify the proposition via Definition \ref{defn:DAHA}.
\end{proof}

Until the end of this paper we adopt the following conventions: Let $k_0,k_1,k_2,k_3$ denote any nonzero scalars in $\F$. Let $\{m_i\}_{i=0}^\infty$ denote the $\F$-basis for $M(k_0,k_1,k_2,k_3)$ from Proposition \ref{prop:P}(i). 
In addition, we set the following parameters associated with $k_0,k_1,k_2,k_3,q$:
\begin{align}
\phi_i
&=\left\{
\begin{array}{ll}
(1-q^i) (1-k_0^2 q^i)
\qquad 
&\hbox{if $i$ is even},
\\
(k_0 k_1^{-1} k_3 q^i-k_2)(k_0 k_1^{-1} k_3 q^i-k_2^{-1})
\qquad 
&\hbox{if $i$ is odd};
\end{array}
\right.
\label{phi}
\\
\varrho_i
&=\left\{
\begin{array}{ll}
(1-q^i) (1-k_0^2 q^i)
\qquad 
&\hbox{if $i$ is even},
\\
(k_0 k_1 k_3 q^i-k_2)(k_0 k_1 k_3 q^i-k_2^{-1})
\qquad 
&\hbox{if $i$ is odd};
\end{array}
\right.
\label{varphi}
\\
\chi_i
&=\left\{
\begin{array}{ll}
(1-q^i) (1-k_0^2 q^i)
\qquad 
&\hbox{if $i$ is even},
\\
(k_0 k_1 k_3^{-1} q^i-k_2)(k_0 k_1 k_3^{-1} q^i-k_2^{-1})
\qquad 
&\hbox{if $i$ is odd};
\end{array}
\right.
\label{psi}
\\
\psi_i
&=\left\{
\begin{array}{ll}
(1-q^i) (1-k_1^2 q^i)
\qquad 
&\hbox{if $i$ is even},
\\
(k_0 k_1 k_2 q^i-k_3)(k_0 k_1 k_2 q^i-k_3^{-1})
\qquad 
&\hbox{if $i$ is odd}
\end{array}
\right.
\label{varrho}
\end{align}
for all $i\in \Z$.

Define 
\begin{align}
X&=t_3t_0,
\label{X}
\\
Y&=t_0t_1.
\label{Y}
\end{align}

\begin{lem}\label{lem:XYinP}
\begin{enumerate}
\item The action of $X$ on $M(k_0,k_1,k_2,k_3)$ is as follows:
\begin{align*}
(1-k_0 k_3 q^{2\lceil \frac{i}{2}\rceil} X^{(-1)^{i-1}}) m_{i}=
\left\{
\begin{array}{ll}
0
\qquad 
&\hbox{if $i=0$},
\\
\varrho_i
m_{i-1}
\qquad 
&\hbox{if $i=1,2,3,\ldots$}.
\end{array}
\right.
\end{align*}

\item The action of $Y$ on $M(k_0,k_1,k_2,k_3)$ is as follows:
\begin{align*}
(1-k_0 k_1 q^{2\lceil \frac{i}{2}\rceil} Y^{(-1)^{i-1}}) m_i =m_{i+1}
\qquad 
\hbox{for all $i=0,1,2,\ldots$}.
\end{align*}
\end{enumerate}
\end{lem}
\begin{proof}
Evaluate the actions of $X$ and $Y$ on $M(k_0,k_1,k_2,k_3)$ by Proposition \ref{prop:P}.
\end{proof}

To state the universal property for $M(k_0,k_1,k_2,k_3)$ we are now going to give an alternative description of $M(k_0,k_1,k_2,k_3)$. We begin with the following lemma.

\begin{lem}
[Proposition 7.8, \cite{DAHA2013}]
\label{lem:basisH}
The elements
$$
Y^i X^j t_0^k
c_0^\ell c_1^r c_2^s c_3^t
$$
for all integers $i,j,\ell,r,s,t$ and $k\in\{0,1\}$ with $\ell,r,s,t\geq 0$ are an $\F$-basis for $\H_q$.
\end{lem}

Let $I(k_0,k_1,k_2,k_3)$ denote the left ideal of $\H_q$ generated by 
\begin{gather}
t_0-k_0,
\qquad 
t_3-k_3,
\label{I1}
\\
c_1-k_1-k_1^{-1},
\qquad 
c_2-k_2-k_2^{-1}.
\label{I2}
\end{gather}

\begin{lem}\label{lem:H/I}
The $\F$-vector space $\H_q/I(k_0,k_1,k_2,k_3)$ is spanned by 
$$
Y^i+I(k_0,k_1,k_2,k_3)
\qquad 
\hbox{for all $i\in \Z$}.
$$
\end{lem}
\begin{proof}
By Lemma \ref{lem:basisH} the cosets 
$$
Y^i X^j t_0^k
c_0^\ell c_1^r c_2^s c_3^t
+
I(k_0,k_1,k_2,k_3)
$$
for all integers $i,j,\ell,r,s,t$ and $k\in\{0,1\}$ with $\ell,r,s,t\geq 0$ span $\H_q/I(k_0,k_1,k_2,k_3)$. Since $I(k_0,k_1,k_2,k_3)$ contains the elements (\ref{I1}) this yields that
\begin{gather}
c_0-k_0-k_0^{-1},
\qquad 
c_3-k_3-k_3^{-1}
\label{I3}
\end{gather}
are in $I(k_0,k_1,k_2,k_3)$. By (\ref{X}) and since $I(k_0,k_1,k_2,k_3)$ contains (\ref{I1})--(\ref{I3}) it follows that 
\begin{align*}
Y^i+I(k_0,k_1,k_2,k_3)
\qquad 
\hbox{for all $i\in \Z$}
\end{align*}
span $\H_q/I(k_0,k_1,k_2,k_3)$. The lemma follows.
\end{proof}

The $\H_q$-module $M(k_0,k_1,k_2,k_3)$ has the following statement:

\begin{thm}\label{thm:Piso}
There exists a unique $\H_q$-module homomorphism 
$$
\Phi:\H_q/I(k_0,k_1,k_2,k_3)\to M(k_0,k_1,k_2,k_3)
$$ 
that sends $1+I(k_0,k_1,k_2,k_3)$ to $m_0$. Moreover $\Phi$ is an isomorphism.
\end{thm}
\begin{proof}
Consider the $\H_q$-module homomorphism $\Psi:\H_q\to M(k_0,k_1,k_2,k_3)$ that sends $1$ to $m_0$. By Proposition \ref{prop:P} the elements (\ref{I1}) and (\ref{I2}) are in the kernel of $\Psi$. It follows that $I(k_0,k_1,k_2,k_3)$ is contained in the kernel of $\Psi$. Therefore there exists an $\H_q$-module homomorphism $\H_q/I(k_0,k_1,k_2,k_3)\to M(k_0,k_1,k_2,k_3)$ that sends $1+I(k_0,k_1,k_2,k_3)$ to $m_0$. The existence follows. Since the $\H_q$-module $\H_q/I(k_0,k_1,k_2,k_3)$ is generated by $1+I(k_0,k_1,k_2,k_3)$ the uniqueness follows.

By Lemma \ref{lem:XYinP}(ii) the homomorphism $\Phi$ sends 
\begin{gather}\label{Yi+I}
\prod_{h=0}^{i-1}
(1-k_0 k_1 q^{2\lceil \frac{h}{2}\rceil} Y^{(-1)^{h-1}})
+I(k_0,k_1,k_2,k_3)
\end{gather}
to $m_i$ for all $i=0,1,2,\ldots$. Since $\{m_i\}_{i=0}^\infty$ are linearly independent, the cosets (\ref{Yi+I}) for all $i=0,1,2,\ldots$ are linearly independent. By Lemma \ref{lem:H/I} the cosets  (\ref{Yi+I}) for all $i=0,1,2,\ldots$ span $\H_q/I(k_0,k_1,k_2,k_3)$. Therefore those cosets form an $\F$-basis for $\H_q/I(k_0,k_1,k_2,k_3)$ and it follows that $\Phi$ is an isomorphism. The result follows.
\end{proof}

In light of Theorem \ref{thm:Piso} the $\H_q$-module $M(k_0,k_1,k_2,k_3)$ has the following universal property:

\begin{thm}\label{thm:universal}
If $V$ is an $\H_q$-module which contains a vector $v$ satisfying
\begin{gather*}
t_0 v=k_0 v, 
\qquad 
t_3 v=k_3 v,
\\
c_1 v=(k_1+k_1^{-1})v,
\qquad 
c_2 v=(k_2+k_2^{-1})v,
\end{gather*}
then there exists a unique $\H_q$-module homomorphism $M(k_0,k_1,k_2,k_3)\to V$ that sends $m_0$ to $v$.
\end{thm}


We finish this section with a comment on the $\H_q$-module $M(k_0,k_1,k_2,k_3)$ and the polynomial representation of $\H_q$. 
Recall that the $\H_q$-module $P(k_0,k_1,k_2,k_3)$ corresponding to the polynomial representation of $\H_q$ \cite{DAHA_book,DAHA&OP_book} is the ring of Laurent polynomials in one variable $z$ on which 
the actions of $t_0,t_1,t_2,t_3$ are as follows:
\begin{eqnarray*}
t_0 : f(z) &\mapsto &
k_0 f(q^2 z^{-1})+\frac{k_0+k_0^{-1}-(k_1+k_1^{-1})q z^{-1}}{1-q^2 z^{-2}}(f(z)-f(q^2 z^{-1})),
\\
t_1 : f(z)
&\mapsto &
\frac{k_1+k_1^{-1}-(k_0+k_0^{-1})q z^{-1}}{1-q^2 z^{-2}}f(z)
+
\frac{k_1+k_1^{-1}-k_0 q z^{-1} - k_0^{-1} q^{-1}z}{1-q^{-2} z^2} f(q^2 z^{-1}),
\\
t_2 : f(z) &\mapsto & 
\frac{k_2+k_2^{-1}-(k_3+k_3^{-1})z}{1-z^2} f(z)
+
\frac{k_3 z+k_3^{-1} z^{-1}-k_2-k_2^{-1}}{1-z^2} f(z^{-1}),
\\
t_3 : f(z) &\mapsto & 
k_3 f(z^{-1})+\frac{k_3+k_3^{-1}-(k_2+k_2^{-1})z}{1-z^2} (f(z)-f(z^{-1})).
\end{eqnarray*}
Observe that $t_0,t_3,c_1,c_2$ map $1$ to $k_0,k_3,k_1+k_1^{-1},k_2+k_2^{-1}$ respectively. It follows from Theorem \ref{thm:universal} that there exists a unique $\H_q$-module homomorphism 
\begin{gather}\label{H->P}
M(k_0,k_1,k_2,k_3)\to P(k_0,k_1,k_2,k_3)
\end{gather}
that sends $m_0$ to $1$. The element $Y$ acts on $P(k_0,k_1,k_2,k_3)$ as multiplication by $q^{-1}z$. By Lemma \ref{lem:XYinP}(ii) the homomorphism (\ref{H->P}) sends $m_i$ to 
\begin{gather}\label{e:Pbasis}
\prod_{h=0}^{i-1}(1-k_0k_1 q^{2\lceil \frac{h}{2}\rceil} q^{(-1)^h} z^{(-1)^{h-1}})
\qquad 
\hbox{for all $i=0,1,2,\ldots$}.
\end{gather}
Since the Laurent polynomials (\ref{e:Pbasis}) form an $\F$-basis for $P(k_0,k_1,k_2,k_3)$ it follows that (\ref{H->P}) is an isomorphism.

\section{Simultaneous eigenvectors of $t_0$ and $t_3$ in $\H_q$-modules}\label{s:eigenvector}

As suggested by Theorem \ref{thm:universal}, we examine the existence of the simultaneous eigenvectors of $t_0$ and $t_3$ in finite-dimensional irreducible $\H_q$-modules and similar issues in this section.

\begin{lem}\label{lem:Schur}
Assume that $\F$ is algebraically closed. If $V$ is a finite-dimensional irreducible $\H_q$-module, then each central element of $\H_q$ acts on $V$ as scalar multiplication.
\end{lem}
\begin{proof}
Apply Schur's lemma to $\H_q$.
\end{proof}

\begin{lem}\label{lem:theta}
Let $\mu$ denote a nonzero scalar in $\F$ and 
\begin{gather}\label{theta}
\theta_i
=\left\{
\begin{array}{ll}
\mu q^i
\quad 
&\hbox{if $i$ is even},
\\
\mu^{-1} q^{-i-1}
\quad 
&\hbox{if $i$ is odd}
\end{array}
\right.
\end{gather}
for all $i\in \Z$. Then the following hold:
\begin{enumerate}
\item If $i=j\bmod{2}$ then $\theta_i=\theta_j$ if and only if $q^i=q^j$.

\item If $i\not=j\bmod{2}$ then $\theta_i=\theta_j$ if and only if $\mu^2=q^{-i-j-1}$. 

\item For any $k\in \Z$ either of $\{\theta_i\}_{i=k}^{\infty}$ and $\{\theta_i\}_{i=k}^{-\infty}$ contains infinitely many values provided that $q$ is not a root of unity.
\end{enumerate}
\end{lem}
\begin{proof}
(i), (ii): By (\ref{theta}) we have 
\begin{gather*}
\theta_i-\theta_j
=
\left\{
\begin{array}{ll}
\mu(q^i-q^j)
\qquad 
&\hbox{if $i$ and $j$ are even},
\\
\mu^{-1} q^{-1}(q^{-i}-q^{-j})
\qquad 
&\hbox{if $i$ and $j$ are odd},
\\
\mu^{-1} q^{-j-1}(\mu^2 q^{i+j+1}-1)
\qquad 
&\hbox{if $i$ is even and $j$ is odd},
\\
\mu^{-1} q^{-i-1}(1-\mu^2 q^{i+j+1})
\qquad 
&\hbox{if $i$ is odd and $j$ is even}.
\end{array}
\right.
\end{gather*}
Therefore (i) and (ii) follow.

(iii): Suppose that $q$ is not a root of unity. If the scalars $\{\theta_i\}_{i\in \Z}$ are mutually distinct, there is nothing to prove. Assume the contrary. It follows from (i) and (ii) that for any distinct $i,j\in \Z$, the scalars $\theta_i$ and $\theta_j$ are equal whenever $i+j$ is equal to a constant. Therefore (iii) follows.
\end{proof}

Recall the element $X$ of $\H_q$ from  (\ref{X}).

\begin{lem}\label{lem:sim_equation}
The following equations hold in $\H_q$:
\begin{enumerate}
\item $Xt_0-t_0X^{-1}=X c_0-c_3$.

\item
$q^{-1} X^{-1} t_2-q t_2 X 
=q^{-1} X^{-1} c_2- c_1$.
\end{enumerate}
\end{lem}
\begin{proof}
(i): Observe that the left-hand side of (i) is equal to 
$$
Xt_0-t_3^{-1}.
$$
Using (\ref{ci}) yields that the right-hand side of (i) is equal to the above. Therefore (i) follows.

(ii): Using (\ref{t0t1t2t3}) yields that
\begin{gather}\label{t2X}
t_2 X=q^{-1} t_1^{-1}.
\end{gather}
By (\ref{t2X}) the left-hand side of (ii) is equal to 
\begin{gather}\label{LHS:t2X}
q^{-1} X^{-1} t_2-t_1^{-1}.
\end{gather} 
By (\ref{ci}) the right-hand side of (ii) is equal to 
\begin{gather}\label{RHS:t2X}
q^{-1} X^{-1} t_2+q^{-1} X^{-1} t_2^{-1}-t_1-t_1^{-1}.
\end{gather} 
Using (\ref{t2X}) the element (\ref{RHS:t2X}) is equal to (\ref{LHS:t2X}). Therefore (ii) follows.
\end{proof}

\begin{prop}\label{prop:X_eigenvector}
Assume that $\F$ is algebraically closed and $q$ is not a root of unity. If $V$ is a finite-dimensional irreducible $\H_q$-module, then the following {\rm (i)} or {\rm (ii)} holds:
\begin{enumerate}
\item $t_3$ and $t_0$ have a simultaneous eigenvector in $V$.

\item $t_1$ and $t_2$ have a simultaneous eigenvector in $V$.
\end{enumerate}
\end{prop}
\begin{proof}
Since $\F$ is algebraically closed and $V$ is finite-dimensional, there exists an eigenvalue $\mu$ of $X$ in $V$. Since $X$ is invertible in $\H_q$ the scalar $\mu$ is nonzero. Consider the corresponding scalars $\{\theta_i\}_{i\in \Z}$ given in (\ref{theta}). 
By Lemma \ref{lem:theta}(iii) there are infinitely many values among $\{\theta_i\}_{i=0}^{-\infty}$. Since $V$ is finite-dimensional there exists an integer $k\leq 0$ such that $\theta_{k-1}$ is not an eigenvalue of $X$ but $\theta_k$ is an eigenvalue of $X$ in $V$. Let $W$ denote the $\theta_k$-eigenspace of $X$ in $V$.

Suppose that $k$ is even. 
Pick any $v\in W$. Note that $\theta_k^{-1}=\theta_{k-1}$.  Applying $v$ to either side of Lemma \ref{lem:sim_equation}(i) yields that 
\begin{gather}\label{Xt0v}
(X-\theta_{k-1}) t_0 v= ( \theta_k c_0 - c_3) v.
\end{gather}
By Lemma \ref{lem:Schur} the right-hand side of (\ref{Xt0v}) is a scalar multiple of $v$.  Left multiplying either side of (\ref{Xt0v})  by $X-\theta_k$ yields that 
$$
(X-\theta_{k-1})(X-\theta_k)t_0v=0.
$$
Since $\theta_{k-1}$ is not an eigenvalue of $X$ in $V$  it follows that $(X-\theta_k)t_0v=0$ and hence $t_0 v\in W$. This shows that $W$ is $t_0$-invariant. Since $\F$ is algebraically closed there exists an eigenvector $w$ of $t_0$ in  $W$. Since $X=t_3t_0$ and $t_0$ is invertible in $\H_q$ it follows that $w$ is also an eigenvector of $t_3$. Therefore (i) follows.

Suppose that $k$ is odd. Pick any $v\in W$.  Note that $\theta_k=q^{-2} \theta_{k-1}^{-1}$. Applying $v$ to Lemma \ref{lem:sim_equation}(ii) yields that 
\begin{gather}\label{X-1t2v}
q^{-1} (X^{-1}-\theta_{k-1}^{-1}) t_2 v
=(q^{-1}\theta_k^{-1} c_2-c_1) v.
\end{gather}
By Lemma \ref{lem:Schur} the right-hand side of (\ref{X-1t2v}) is a scalar multiple of $v$. 
Left multiplying either side of (\ref{X-1t2v})  by $X-\theta_k$ yields that 
$$
(X^{-1}-\theta_{k-1}^{-1})(X-\theta_k)t_2v=0.
$$
Since $\theta_{k-1}^{-1}$ is not an eigenvalue of $X^{-1}$ in $V$ it follows that $(X-\theta_k)t_2v=0$ and hence $t_2 v\in W$. This shows that $W$ is $t_2$-invariant. Since $\F$ is algebraically closed there exists an eigenvector $w$ of $t_2$ in $W$. Since $X^{-1}=q t_1 t_2$ by (\ref{t0t1t2t3}) and $t_2$ is invertible in $\H_q$ it follows that $w$ is also an eigenvector of $t_1$. Therefore (ii) follows. 
\end{proof}

\begin{prop}\label{prop:Y_eigenvector}
Assume that $\F$ is algebraically closed and $q$ is not a root of unity. If $V$ is a finite-dimensional irreducible $\H_q$-module, then the following {\rm (i)} or {\rm (ii)} holds:
\begin{enumerate}
\item $t_0$ and $t_1$ have a simultaneous eigenvector in $V$.

\item $t_2$ and $t_3$ have a simultaneous eigenvector in $V$.
\end{enumerate}
\end{prop}
\begin{proof}
The proposition follows by applying Proposition \ref{prop:X_eigenvector} to the $\H_q$-module $V^{1\bmod{4}}$.
\end{proof}

\section{Proof of Theorem \ref{thm:iso}}\label{section:iso}

Throughout this section we use the following conventions: 
Let $d\geq 1$ denote an odd integer. Assume that $k_0^2=q^{-d-1}$. 
Let $\{v_i\}_{i=0}^d$ denote the $\F$-basis for the $\H_q$-module $E(k_0,k_1,k_2,k_3)$ from Proposition \ref{prop:E}(i).
Define 
$$
N(k_0,k_1,k_2,k_3)
$$
to be the $\H_q$-submodule of $M(k_0,k_1,k_2,k_3)$ generated by $m_{d+1}$.

\begin{lem}\label{lem:XYinE}
\begin{enumerate}
\item The action of $X$ on $E(k_0,k_1,k_2,k_3)$ is as follows:
\begin{align*}
(1-k_0 k_3 q^{2\lceil \frac{i}{2}\rceil}X^{(-1)^{i-1}}) v_i
=
\left\{
\begin{array}{ll}
0
\qquad 
&\hbox{if $i=0$},
\\
\varrho_i
v_{i-1}
\qquad 
&\hbox{if $i=1,2,\ldots,d$}.
\end{array}
\right.
\end{align*}

\item The action of $Y$ on $E(k_0,k_1,k_2,k_3)$ is as follows:
\begin{align*}
(1-k_0 k_1 q^{2\lceil \frac{i}{2}\rceil} Y^{(-1)^{i-1}}) v_i
=
\left\{
\begin{array}{ll}
0
\qquad 
&\hbox{if $i=d$},
\\
v_{i+1}
\qquad 
&\hbox{if $i=0,1,\ldots,d-1$}.
\end{array}
\right.
\end{align*}
\end{enumerate}
\end{lem}
\begin{proof}
Recall $X$ and $Y$ from (\ref{X}) and (\ref{Y}). Evaluate the actions of $X$ and $Y$ on $E(k_0,k_1,k_2,k_3)$ by Proposition \ref{prop:E}.
\end{proof}

\begin{lem}\label{lem:N_E}
$\{m_i\}_{i=d+1}^\infty$ is an $\F$-basis for $N(k_0,k_1,k_2,k_3)$.
\end{lem}
\begin{proof}
Since $d$ is odd and $k_0^2=q^{-d-1}$ it follows from Proposition \ref{prop:P} that $N(k_0,k_1,k_2,k_3)$ has the $\F$-basis $\{m_i\}_{i=d+1}^\infty$.
\end{proof}

\begin{lem}\label{lem:E}
There exists a unique $\H_q$-module isomorphism 
$$
M(k_0,k_1,k_2,k_3)/N(k_0,k_1,k_2,k_3)\to E(k_0,k_1,k_2,k_3)
$$ 
that sends $m_i+N(k_0,k_1,k_2,k_3)$ to $v_i$ for all $i=0,1,\ldots,d$.
\end{lem}
\begin{proof}
By Proposition \ref{prop:E} the vector $v_0$ satisfies 
\begin{gather*}
t_0 v_0=k_0 v_0, 
\qquad 
t_3 v_0 =k_3 v_0,
\\
c_1 v_0 =(k_1+k_1^{-1})v_0,
\qquad 
c_2 v_0 =(k_2+k_2^{-1})v_0.
\end{gather*}
Hence there exists a unique $\H_q$-module homomorphism 
\begin{gather}\label{M->E}
M(k_0,k_1,k_2,k_3)\to E(k_0,k_1,k_2,k_3)
\end{gather}
that sends $m_0$ to $v_0$ by Theorem \ref{thm:universal}. 
Comparing Lemma \ref{lem:XYinP}(ii) with Lemma \ref{lem:XYinE}(ii) the image of $m_i$ under (\ref{M->E}) is $v_i$ for all $i=0,1,\ldots,d$ and the image of $m_{d+1}$ under (\ref{M->E}) is zero. Hence $N(k_0,k_1,k_2,k_3)$ is contained in the kernel of (\ref{M->E}). Therefore (\ref{M->E}) induces the $\H_q$-module homomorphism 
\begin{gather}\label{M/N->E}
M(k_0,k_1,k_2,k_3)/N(k_0,k_1,k_2,k_3)\to E(k_0,k_1,k_2,k_3)
\end{gather}
that sends $m_i+N(k_0,k_1,k_2,k_3)$ to $v_i$ for all $i=0,1,\ldots,d$. By Lemma \ref{lem:N_E} the cosets $\{m_i+N(k_0,k_1,k_2,k_3)\}_{i=0}^d$ form an $\F$-basis for $M(k_0,k_1,k_2,k_3)/N(k_0,k_1,k_2,k_3)$.
Since $\{v_i\}_{i=0}^d$ is an $\F$-basis for $E(k_0,k_1,k_2,k_3)$, it follows that (\ref{M/N->E}) is an isomorphism.
\end{proof}

\begin{prop}\label{prop:universal_E}
Let $V$ denote an $\H_q$-module. 
If the element
\begin{gather}\label{e:universal_E}
\prod_{i=0}^{d}
(1-k_0 k_1q^{2\lceil \frac{i}{2}\rceil}Y^{(-1)^{i-1}})
\end{gather}
vanishes at some vector $v$ of $V$ and there is an $\H_q$-module homomorphism $M(k_0,k_1,k_2,k_3)\to V$ that sends $m_0$ to $v$, then there exists an $\H_q$-module homomorphism $E(k_0,k_1,k_2,k_3)\to V$ that sends $v_0$ to $v$.
\end{prop}
\begin{proof}
Let $\Phi$ denote the $\H_q$-module homomorphism $M(k_0,k_1,k_2,k_3)\to V$ sends $m_0$ to $v$. Since (\ref{e:universal_E}) vanishes at $v$ and by Lemma \ref{lem:XYinP}(ii) the vector $m_{d+1}$ is in the kernel of $\Phi$. It follows that there is an $\H_q$-module homomorphism $M(k_0,k_1,k_2,k_3)/N(k_0,k_1,k_2,k_3)\to V$ that sends $m_0+N(k_0,k_1,k_2,k_3)$ to $v$. Combined with Lemma \ref{lem:E} the proposition follows.
\end{proof}

\begin{lem}\label{lem:irr_E}
If the $\H_q$-module $E(k_0,k_1,k_2,k_3)$ is irreducible then the following conditions hold:
\begin{enumerate}
\item $q^i\not=1$ for all $i=2,4,\ldots,d-1$.

\item $k_0^2\not=q^{-i}$ for all $i=2,4,\ldots,d-1$.

\item $k_0k_1k_2k_3, k_0k_1k_2^{-1}k_3\not=q^{-i}$ for all $i=1,3,\ldots,d$.
\end{enumerate}
\end{lem}
\begin{proof}
By (\ref{varphi}) the conditions (i)--(iii) hold if and only if $\varrho_i\not=0$ for all $i=1,2,\ldots,d$. Suppose that there is a $k\in\{1,2,\ldots,d\}$ such that $\varrho_k=0$. Let $W$ denote the $\F$-subspace of $E(k_0,k_1,k_2,k_3)$ spanned by $\{v_i\}_{i=k}^d$. Using Proposition \ref{prop:E} yields that $W$ is invariant under $\{t_i^{\pm 1}\}_{i=0}^3$. Hence $W$ is a proper $\H_q$-submodule of $E(k_0,k_1,k_2,k_3)$, a contradiction to the irreducibility of $E(k_0,k_1,k_2,k_3)$. The lemma follows.
\end{proof}

\begin{prop}\label{prop:iso2}
The $\H_q$-module $E(k_0,k_1,k_2,k_3)$ is isomorphic to $E(k_0,k_1^{-1},k_2,k_3)$.
Moreover the vectors 
$$
w_i=
\displaystyle
\prod_{h=0}^{i-1}
(1-k_0 k_1^{-1} q^{2\lceil \frac{h}{2}\rceil} Y^{(-1)^{h-1}}) v_0
\qquad 
\hbox{for all $i=0,1,\ldots,d$}
$$
form an $\F$-basis for $E(k_0,k_1,k_2,k_3)$ and 
\begin{align*} 
(1-k_0 k_3 q^{2\lceil \frac{i}{2}\rceil} X^{(-1)^{i-1}}) w_i
&=
\left\{
\begin{array}{ll}
0
\qquad 
&\hbox{if $i=0$},
\\
\phi_i
w_{i-1}
\qquad 
&\hbox{if $i=1,2,\ldots,d$};
\end{array}
\right.
\\
(1-k_0 k_1^{-1} q^{2\lceil \frac{i}{2}\rceil} Y^{(-1)^{i-1}}) w_i
&=
\left\{
\begin{array}{ll}
0
\qquad 
&\hbox{if $i=d$},
\\
w_{i+1}
\qquad 
&\hbox{if $i=0,1,\ldots,d-1$}.
\end{array}
\right.
\end{align*}
\end{prop}
\begin{proof}
Let $\{u_i\}_{i=0}^d$ denote the $\F$-basis for $E(k_0,k_1^{-1},k_2,k_3)$ obtained from Proposition \ref{prop:E}(i). 
To see the proposition, it suffices to show that there exists an $\H_q$-module homomorphism $E(k_0,k_1,k_2,k_3)\to E(k_0,k_1^{-1},k_2,k_3)$ that sends $w_i$ to $u_i$ for all $i=0,1,\ldots,d$.

By Proposition \ref{prop:E} the vector $u_0$ satisfies the following equations:
\begin{gather*}
t_0 u_0=k_0 u_0,
\qquad 
t_3 u_0=k_3 u_0,
\\
c_1 u_0=(k_1+k_1^{-1}) u_0,
\qquad 
c_2 u_0=(k_2+k_2^{-1}) u_0.
\end{gather*}
Hence there exists a unique $\H_q$-module homomorphism
\begin{gather*}
M(k_0,k_1,k_2,k_3)\to E(k_0,k_1^{-1},k_2,k_3)
\end{gather*} 
that sends $m_0$ to $u_0$ by Theorem \ref{thm:universal}. It follows from Lemma \ref{lem:XYinE}(ii) that 
\begin{gather}\label{w_0=0}
\prod_{i=0}^d(1-k_0k_1^{-1}q^{2 \lceil \frac{i}{2}\rceil} Y^{(-1)^{i-1}})
\end{gather}
vanishes at $u_0$. Since $d$ is odd and $k_0^2=q^{-d-1}$ the $i^{\,{\rm th}}$ term of (\ref{w_0=0}) is equal to $-k_0k_1^{-1} q^{2 \lceil \frac{i}{2}\rceil} Y^{(-1)^{i-1}}$ times the $(d-i)^{\,{\rm th}}$ term of (\ref{e:universal_E}) 
for all $i=0,1,\ldots,d$.
Hence there exists an $\H_q$-module homomorphism
\begin{gather}\label{E(k1inv)->E}
E(k_0,k_1,k_2,k_3)\to E(k_0,k_1^{-1},k_2,k_3)
\end{gather} 
that sends $v_0$ to $u_0$ by Proposition \ref{prop:universal_E}. By the construction of $\{w_i\}_{i=0}^d$ the homomorphism (\ref{E(k1inv)->E}) maps $w_i$ to $u_i$ for all $i=0,1,\ldots,d$. The proposition follows.
\end{proof}

\begin{lem}\label{lem:irr2}
If the $\H_q$-module $E(k_0,k_1,k_2,k_3)$ is irreducible then the following conditions hold:
\begin{enumerate}
\item $q^i\not=1$ for all $i=2,4,\ldots,d-1$.

\item $k_0^2\not=q^{-i}$ for all $i=2,4,\ldots,d-1$.

\item $k_0k_1k_2k_3, k_0k_1^{-1}k_2k_3, k_0k_1k_2^{-1}k_3, k_0k_1k_2 k_3^{-1}\not=q^{-i}$ for all $i=1,3,\ldots,d$.
\end{enumerate}
\end{lem}
\begin{proof}
Suppose that the $\H_q$-module $E(k_0,k_1,k_2,k_3)$ is irreducible. 
By Proposition \ref{prop:iso2} the $\H_q$-module $E(k_0,k_1^{-1},k_2,k_3)$ is irreducible.
Applying Lemma \ref{lem:irr_E} to $E(k_0,k_1^{-1},k_2,k_3)$ yields that $k_0k_1^{-1}k_2k_3,k_0k_1^{-1}k_2^{-1}k_3\not=q^{-i}$ for all $i=1,3,\ldots,d$. Since $k_0^2=q^{-d-1}$ the conditions $k_0k_1^{-1}k_2^{-1}k_3\not=q^{-i}$ for all $i=1,3,\ldots,d$ are equivalent to $k_0k_1k_2k_3^{-1}\not=q^{-i}$ for all $i=1,3,\ldots,d$. The lemma follows. 
\end{proof}

Define
\begin{align*}
R &= \prod_{h=1}^{d}
(1-k_0 k_3 q^{2\lceil \frac{h}{2}\rceil} X^{(-1)^{h-1}}),
\\
S_i &= 
\prod_{h=1}^{d-i}
(1-k_0 k_1^{-1} q^{2\lceil \frac{h-1}{2}\rceil} Y^{(-1)^h})
\qquad 
\hbox{for all $i=0,1,\ldots,d$}.
\end{align*}
It follows from Lemma \ref{lem:XYinE}(i) that $R v$ is a scalar multiple of $v_0$ for all $v\in E(k_0,k_1,k_2,k_3)$. In particular, for any $i,j\in \{0,1,\ldots,d\}$ there exists a unique $L_{ij}\in \F$ such that 
\begin{gather}\label{Lij}
R S_i v_j= L_{ij} v_0.
\end{gather}
By Lemma \ref{lem:XYinE}(ii) the scalars 
\begin{align}
L_{ij} &=0 \qquad 
\hbox{for all $0\leq i<j\leq d$};
\label{L:lower}\\
L_{ij} &=
\left\{
\begin{array}{ll}
k_1^2
q^{i+j-d-1}
(L_{i-1,j-1}-L_{i,j-1})+L_{i,j-1}
&\hbox{if $i=j \bmod{2}$},
\\
(1-q^{j-i-1}) L_{i,j-1}
+L_{i-1,j}-L_{i-1,j-1}
&\hbox{if $i\not=j \bmod{2}$}
\end{array}
\right.
\label{L:recurrence}
\end{align}
for all $i,j\in\{1,2,\ldots,d\}$. 
By Proposition \ref{prop:iso2} the scalars 
\begin{gather}\label{L:init}
L_{i0}=
\prod_{h=1}^{\lfloor \frac{i}{2}\rfloor}
(1-
q^{d-2h+1})
\prod_{h=1}^{\lceil \frac{i}{2}\rceil}
(1-k_3^2 
q^{2-2h}
)
\prod_{h=1}^{d-i} \phi_h
\qquad 
\hbox{for all $i=0,1,\ldots,d$}.
\end{gather}
Solving the recurrence relation (\ref{L:recurrence}) with the initial conditions (\ref{L:lower}) and (\ref{L:init}) yields that 
\begin{align}
L_{ij}
&=
\displaystyle
q^{\lfloor\frac{j}{2}\rfloor(d-4\lfloor \frac{i}{2}\rfloor+2\lfloor\frac{j}{2}\rfloor-1)}
\prod_{h=1}^{\lfloor\frac{i-j}{2}\rfloor}
(1-q^{d-2h+1})
\prod_{h=1}^{d-i} \phi_h
\prod_{h=1}^{\lceil\frac{j}{2}\rceil}
\varrho_{2h-1}
\prod_{h=1}^{\lfloor\frac{j}{2}\rfloor}
\varrho_{2(\lfloor \frac{i}{2}\rfloor-h+1)}
\label{e:Lij}
\\
&\qquad\times 
\;
\left\{
\begin{array}{ll}
\displaystyle
\prod_{h=1}^{\lceil\frac{i-j}{2}\rceil}
(1-k_3^2 q^{2-2h})
\qquad 
&\hbox{if $i$ is odd or $j$ is even},
\\
\displaystyle
q^{d-2i+2j-1}
\varrho_{i-j+1}
\prod_{h=1}^{\lfloor\frac{i-j}{2}\rfloor}
(1-k_3^2 q^{2-2h})
\qquad 
&\hbox{if $i$ is even and $j$ is odd}
\end{array}
\right.
\notag
\end{align}
for all $0\leq j\leq i\leq d$.

We now show that the converse of Lemma \ref{lem:irr2} is true.

\begin{thm}\label{thm:irr}
The $\H_q$-module $E(k_0,k_1,k_2,k_3)$ is irreducible if and only if the following conditions hold:
\begin{enumerate}
\item $q^i\not=1$ for all $i=2,4,\ldots,d-1$.

\item $k_0^2\not=q^{-i}$ for all $i=2,4,\ldots,d-1$.

\item $k_0k_1k_2k_3, k_0k_1^{-1}k_2k_3, k_0k_1k_2^{-1}k_3, k_0k_1k_2 k_3^{-1}\not=q^{-i}$ for all $i=1,3,\ldots,d$.
\end{enumerate}
\end{thm}
\begin{proof}
By Lemma \ref{lem:irr2} it remains to prove the ``if'' part. Suppose that (i)--(iii) hold. To see the irreducibility of $E(k_0,k_1,k_2,k_3)$, we let $W$ denote a nonzero $\H_q$-submodule of $E(k_0,k_1,k_2,k_3)$ and show that $W= E(k_0,k_1,k_2,k_3)$.

Pick any nonzero $w\in W$. Let $[w]$ denote the coordinate vector of $w$ relative to $\{v_i\}_{i=0}^d$. Let $L$ denote the $(d+1)\times (d+1)$ matrix indexed by $0,1,\ldots,d$ whose $(i,j)$-entry $L_{ij}$ is defined as (\ref{Lij}) for all $i,j\in\{0,1,\ldots,d\}$. 
By (\ref{L:lower}) the square matrix $L$ is lower triangular. By (\ref{e:Lij}) the diagonal entries of $L$ are 
$$
L_{ii}=
q^{\lfloor \frac{i}{2}\rfloor(d-2\lfloor \frac{i}{2}\rfloor-1)}
\prod_{h=1}^{d-i}\phi_h
\prod_{h=1}^{\lceil \frac{i}{2}\rceil} \varrho_{2h-1}
\prod_{h=1}^{\lfloor \frac{i}{2}\rfloor} \varrho_{2(\lfloor \frac{i}{2}\rfloor-h+1)}
$$
for all $i=0,1,\ldots,d$. 
Recall the parameters $\{\phi_i\}_{i\in \Z}$ and $\{\varrho_i\}_{i\in \Z}$ from  (\ref{phi}) and (\ref{varphi}). By (i)--(iii) each of $\{\phi_i\}_{i=1}^d$ and $\{\varrho_i\}_{i=1}^d$ is nonzero. 
Hence $L$ is nonsingular.  Observe that $RS_i w\in W$ is scalar multiplication of $v_0$ by the $i$-entry of $L[w]$. Since $[w]$ is nonzero it follows that $L[w]$ is nonzero and this implies that $v_0\in W$. By Lemma \ref{lem:XYinE}(ii) the $\H_q$-module $E(k_0,k_1,k_2,k_3)$ is generated by $v_0$. Therefore $W=E(k_0,k_1,k_2,k_3)$. The result follows.
\end{proof}

We are now ready to give our proof for Theorem  \ref{thm:iso}.

\medskip

\noindent{\it Proof of Theorem \ref{thm:iso}.}
(i): Immediate from Proposition \ref{prop:iso2}.

(ii): Let $\{u_i\}_{i=0}^d$ denote the $\F$-basis for $E(k_0,k_1,k_2^{-1},k_3)$ obtained from Proposiiton \ref{prop:E}(i). By Proposition \ref{prop:E} there is an $\H_q$-module isomorphism $E(k_0,k_1,k_2,k_3)\to E(k_0,k_1,k_2^{-1},k_3)$ that sends $v_i$ to $u_i$ for all $i=0,1,\ldots,d$. Therefore (ii) follows.

(iii): Let $\{u_i\}_{i=0}^d$ denote the $\F$-basis for $E(k_0,k_1,k_2,k_3^{-1})$ obtained from Proposition \ref{prop:E}(i). Recall the parameters $\{\chi_i\}_{i\in \Z}$ from (\ref{psi}). It is straightforward to verify that each of 
\begin{align*}
w_i&= k_3^2 \chi_i  u_{i-1}
+
(1-k_3^2) u_i
- u_{i+1}
\qquad 
\hbox{for $i=1,3,\ldots,d-2$},
\\
w_d&=k_3^2 \chi_d u_{d-1}+(1-k_3^2)u_d.
\end{align*}
is a $k_3$-eigenvector of $t_3$ in $E(k_0,k_1,k_2,k_3^{-1})$.  Since the $\H_q$-module $E(k_0,k_1,k_2,k_3)$ is irreducible it follows from Theorem \ref{thm:irr} that $\chi_i\not=0$ for all $i=1,2,\ldots,d$ and $q^i\not=1$ for all $i=2,4,\ldots,d-1$. 
Thus we may set 
$$
v=
\sum_{i=1}^{\frac{d+1}{2}}
\prod_{h=1}^{i-1}
\frac{k_3^{-2}-q^{2h-2i}}{1-q^{2h-2i}}
\chi_{2i-2h+1}^{-1}
w_{2i-1}.
$$
By construction the vector $v$ is nonzero and $t_3v=k_3 v$. Using Proposition \ref{prop:E} a direct calculation yields that $t_0 v=k_0 v$. By Theorem \ref{thm:universal} there exists a unique $\H_q$-module homomorphism 
$$
M(k_0,k_1,k_2,k_3)\to E(k_0,k_1,k_2,k_3^{-1})
$$ 
that sends $m_0$ to $v$. By Lemma \ref{lem:XYinE}(ii) the product 
(\ref{e:universal_E}) 
vanishes on $E(k_0,k_1,k_2,k_3^{-1})$. 
Hence there exists a nontrivial $\H_q$-module homomorphism 
\begin{gather}\label{E(k3)->E(k3-1)}
E(k_0,k_1,k_2,k_3)\to E(k_0,k_1,k_2,k_3^{-1})
\end{gather} 
that sends $v_0$ to $v$ by Proposition \ref{prop:universal_E}. Since the $\H_q$-modules $E(k_0,k_1,k_2,k_3)$ and $E(k_0,k_1,k_2,k_3^{-1})$ are irreducible by Theorem \ref{thm:irr}, it follows that (\ref{E(k3)->E(k3-1)}) is an isomorphism. Therefore (iii) follows.
\hfill $\square$

\section{Proof of Theorem \ref{thm:even}}\label{section:even}

In this section we are devoted to proving Theorem \ref{thm:even}.

\begin{thm}\label{thm:onto}
Assume that $\F$ is algebraically closed and $q$ is not a root of unity. Let $d\geq 1$ denote an odd integer. If $V$ is a $(d+1)$-dimensional irreducible $\H_q$-module, then there exist an element $\e\in \Z/4\Z$ and nonzero scalars $k_0,k_1,k_2,k_3\in \F$ with $k_0^2=q^{-d-1}$ such that the $\H_q$-module $E(k_0,k_1,k_2,k_3)^\e$ is isomorphic to $V$.
\end{thm}
\begin{proof}
Suppose that $V$ is a $(d+1)$-dimensional irreducible $\H_q$-module. 
According to Propositions \ref{prop:X_eigenvector} and \ref{prop:Y_eigenvector} we may divide the argument into the following cases (i)--(iv):
\begin{enumerate}
\item[(i)] $t_0,t_3$ have a simultaneous eigenvector and $t_0,t_1$ have a 
simultaneous eigenvector in $V$.

\item[(ii)] $t_1,t_2$ have a simultaneous eigenvector and $t_0,t_1$ have a simultaneous eigenvector in $V$.

\item[(iii)] $t_1,t_2$ have a simultaneous eigenvector and $t_2,t_3$ have a simultaneous eigenvector in $V$.

\item[(iv)] $t_0,t_3$ have a simultaneous eigenvector and $t_2,t_3$ have a simultaneous eigenvector in $V$.
\end{enumerate}

(i): In this case we shall show that there are nonzero $k_0,k_1,k_2,k_3\in \F$ with $k_0^2=q^{-d-1}$ such that the $\H_q$-module $E(k_0,k_1,k_2,k_3)$ is isomorphic to $V$.

 Let $v$ denote a simultaneous eigenvector of $t_0$ and $t_3$ in $V$. Choose the scalars $k_0$ and $k_3$ as the eigenvalues of $t_0$ and $t_3$ corresponding to $v$, respectively. Let $w$ denote a simultaneous eigenvector of $t_0$ and $t_1$ in $V$. Choose $k_1$ as the eigenvalue of $t_1^{-1}$ corresponding to $w$. By Lemma \ref{lem:Schur} the element $c_1$ acts on $V$ as scalar multiplication by $k_1+k_1^{-1}$ and there exists a nonzero scalar $k_2\in \F$ such that 
$c_2$ acts on $V$ as scalar multiplication by $k_2+k_2^{-1}$. 
By Theorem \ref{thm:universal} there exists a unique $\H_q$-module homomorphism
\begin{gather}\label{Hmodule_even1}
M(k_0,k_1,k_2,k_3)\to V
\end{gather}
that sends $m_0$ to $v$.

Let $v_i$ denote the image of $m_i$ under (\ref{Hmodule_even1}) for each $i=0,1,2,\ldots$. 
Suppose that there exists an $i\in \{1,2,\ldots,d\}$ such that $v_i$ is an $\F$-linear combination of $v_0,v_1,\ldots,v_{i-1}$. 
Let $W$ denote the $\F$-subspace of $V$ spanned by $v_0,v_1,\ldots,v_{i-1}$. By construction the dimension of $W$ is less than $d+1$.
By Proposition \ref{prop:P}, $W$ is invariant under $\{t_i^{\pm 1}\}_{i=0}^3$. Hence $W$ is a proper $\H_q$-submodule of $V$, a contradiction to the irreducibility of $V$. 
Therefore 
$
\{v_i\}_{i=0}^d$ is an $\F$-basis for $V$.
Let $\{\varrho_i\}_{i\in \Z}$ denote the parameters (\ref{varphi}) corresponding to the present parameters $k_0,k_1,k_2,k_3$.  Observe that 
\begin{equation}\label{basis_t1}
v_i
\quad 
\hbox{for $i=0,2,\ldots,d-1$},
\quad 
v_1, 
\quad 
v_{i+1}-k_1^2\varrho_i v_{i-1}
\quad
\hbox{for $i=2,4,\ldots,d-1$}
\end{equation}
form an $\F$-basis for $V$. By Proposition \ref{prop:P}(i) the matrix representing $t_1$ with respect to the $\F$-basis (\ref{basis_t1}) for $V$ is 
\begin{gather*}
\begin{pmatrix}
k_1 I_{\frac{d+1}{2}} & \rvline &{\bf 0}
\\
\hline
 k_1^{-1} I_{\frac{d+1}{2}} &\rvline &k_1^{-1} I_{\frac{d+1}{2}} 
\end{pmatrix}.
\end{gather*}
By the rank-nullity theorem the $k_1^{-1}$-eigenspace of $t_1$ in $V$ has dimension $\frac{d+1}{2}$. Hence
$
v_1
$ 
and 
$ 
v_{i+1}-k_1^2\varrho_i v_{i-1}$
for $i=2,4,\ldots,d-1$ form an $\F$-basis for the $k_1^{-1}$-eigenspace of $t_1$ in $V$, 
as well as 
\begin{gather}\label{k1inv-basis}
v_i
\qquad 
\hbox{for all $i=1,3,\ldots,d$}. 
\end{gather}
Since $w$ is a $k_1^{-1}$-eigenvector of $t_1$ in $V$, the vector $w$ is an $\F$-linear combination of (\ref{k1inv-basis}).

For any $u\in V$ let $[u]$ denote the coordinate vector of $u$ relative to $\{v_i\}_{i=0}^d$. We are now particularly concerned with $[v_{d+1}]$ and $[w]$.
Let $a_i$ and $b_i$ denote the $i^{\rm th}$ entries of $[v_{d+1}]$ and $[w]$ for all $i=0,1,\ldots,d$, respectively.
As mentioned earlier the coefficients $b_i=0$ for all $i=0,2,\ldots,d-1$. 
Since $w$ is an eigenvector of $t_0$ in $V$ it follows from Proposition \ref{prop:P}(i) that $b_d\not=0$. Without loss of generality we assume that $$
b_d=1.
$$ 
Observe that the first entry of $[t_0 w]$ is equal to $-a_0k_0^{-1} q^{-d-1}$.
Since $b_0=0$ it follows that 
$$
a_0=0.
$$ 
We now show that $a_i=0$ for all $i=1,2,\ldots,d$. 
Using Lemma \ref{lem:XYinP}(i) yields that $(1-k_0 k_3 q^{d+1} X^{-1}) v_i$ is equal to 
\begin{equation}\label{XinV}
\begin{split}
\left\{
\begin{array}{ll}
(1-q^{d+1})v_0,
\qquad
&\hbox{if $i=0$}, 
\\
(1-k_0^2 k_3^2 q^{d+3}) v_1
-q^{d+1} \varrho_1 v_0
\qquad
&\hbox{if $i=1$},
\\
(1-q^{d-i+1})v_i+q^{d-i+1}\varrho_i v_{i-1}
\qquad 
&\hbox{if $i=2,4,\ldots,d-1$},
\\
(1-k_0^2 k_3^2 q^{d+i+2}) v_i
-q^{d-i+2} \varrho_i(v_{i-1}-\varrho_{i-1} v_{i-2})
\qquad 
&\hbox{if $i=3,5,\ldots,d$}.
\end{array}
\right.
\end{split}
\end{equation}
By Lemma \ref{lem:XYinP}(i) with $i=d+1$ we have
\begin{align}\label{e1:ud+1}
[(1-k_0 k_3 q^{d+1} X^{-1}) v_{d+1}]
&=
\begin{pmatrix}
0 \\
0 \\
\vdots \\
0 \\
\varrho_{d+1}
\end{pmatrix}.
\end{align}
Evaluating the first $d$ entries of the left-hand side of (\ref{e1:ud+1}) by using (\ref{XinV}), we obtain that 
\begin{gather}\label{ai}
\varrho_i a_i
=
\left\{
\begin{array}{ll}
(q^{i-d-2}-1) a_{i-1}
\qquad 
&\hbox{for $i=1,3,\ldots,d$},
\\
(k_0^2 k_3^2 q^{d+i+1}-1) a_{i-1}
\qquad 
&\hbox{for $i=2,4,\ldots,d-1$}.
\end{array}
\right.
\end{gather}
Suppose on the contrary that there exists an $i\in\{1,2,\ldots,d\}$ with $a_i\not=0$ and $a_{i-1}=0$. By (\ref{ai}) the scalar $\varrho_i=0$. Let $W$ denote the $\F$-subspace of $V$ spanned by $v_i,v_{i+1},\ldots,v_d$. It follows from Proposition \ref{prop:P} that $W$ is invariant under $\{t_i^{\pm 1}\}_{i=0}^3$. Hence $W$ is a proper $\H_q$-submodule of $V$, a contradiction to the irreducibility of $V$. Therefore $a_i=0$ for all $i=0,1,\ldots,d$. In other words 
\begin{gather}\label{ud+1=0}
v_{d+1}=0.
\end{gather}

By (\ref{ud+1=0}) the left-hand side of (\ref{e1:ud+1}) is zero. Hence $\varrho_{d+1}=0$. Since $q$ is not a root of unity, this forces that 
$$
k_0^2=q^{-d-1}.
$$
By Lemma \ref{lem:XYinP}(ii) we have 
\begin{gather}\label{e:Yv}
\prod_{i=0}^d(1-k_0 k_1 q^{2\lceil \frac{i}{2}\rceil} Y^{(-1)^{i-1}}) v_0=v_{d+1}.
\end{gather}
By (\ref{ud+1=0}) the right-hand side of (\ref{e:Yv}) is zero. Thus, it follows from  Proposition \ref{prop:universal_E} that there exists a nontrivial $\H_q$-module homomorphism 
\begin{gather}\label{E->V}
E(k_0,k_1,k_2,k_3)\to V.
\end{gather}
Since the $\H_q$-module $V$ is irreducible it follows that (\ref{E->V}) is onto. Since $E(k_0,k_1,k_2,k_3)$ and $V$ are of dimension $d+1$ it follows that (\ref{E->V}) is an isomorphism. Therefore the theorem holds for the case (i).

(ii): Let $\e=1\pmod{4}$. Since $t_0^\e=t_1,t_1^\e=t_2,t_3^\e=t_0$ by Table \ref{Z/4Z-action}, the elements $t_0,t_3$ have a simultaneous eigenvector and $t_0,t_1$ have a 
simultaneous eigenvector in $V^\e$. By (i) there exist nonzero scalars $k_0,k_1,k_2,k_3\in \F$ with $k_0^2=q^{-d-1}$ such that $E(k_0,k_1,k_2,k_3)$ is isomorphic to $V^\e$. Therefore the $\H_q$-module $E(k_0,k_1,k_2,k_3)^{-\e}$ is isomorphic to $V$. The theorem holds for the case (ii).

(iii): Let $\e=2\pmod{4}$. Since $t_0^\e=t_2,t_1^\e=t_3,t_3^\e=t_1$ by Table \ref{Z/4Z-action}, the elements $t_0,t_3$ have a simultaneous eigenvector and $t_0,t_1$ have a 
simultaneous eigenvector in $V^\e$. By (i) there exist nonzero scalars $k_0,k_1,k_2,k_3\in \F$ with $k_0^2=q^{-d-1}$ such that $E(k_0,k_1,k_2,k_3)$ is isomorphic to $V^\e$. Therefore the $\H_q$-module $E(k_0,k_1,k_2,k_3)^{-\e}$ is isomorphic to $V$. The theorem holds for the case (iii).

(iv): Let $\e=3\pmod{4}$. Since $t_0^\e=t_3,t_1^\e=t_0,t_3^\e=t_2$ by Table \ref{Z/4Z-action} the elements $t_0,t_3$ have a simultaneous eigenvector and $t_0,t_1$ have a 
simultaneous eigenvector in $V^\e$. By (i) there exist nonzero scalars $k_0,k_1,k_2,k_3\in \F$ with $k_0^2=q^{-d-1}$ such that $E(k_0,k_1,k_2,k_3)$ is isomorphic to $V^\e$. Therefore the $\H_q$-module $E(k_0,k_1,k_2,k_3)^{-\e}$ is isomorphic to $V$. The theorem holds for the case (iv).
\end{proof}

\begin{lem}\label{lem:injective}
Let $d\geq 1$ denote an odd integer. 
For any nonzero $k_0,k_1,k_2,k_3\in \F$ with $k_0^2=q^{-d-1}$ the determinants of $t_0,t_1,t_2,t_3$ on $E(k_0,k_1,k_2,k_3)$ are $q^{-d-1},1,1,1$ respectively.
\end{lem}
\begin{proof}
It is routine to verify the lemma by using Proposition \ref{prop:E}(i).
\end{proof}

We are now prepared to prove Theorem \ref{thm:even}.

\medskip

\noindent {\it Proof of Theorem \ref{thm:even}.} By Theorems \ref{thm:iso} and \ref{thm:irr} the map $\mathcal E$ is well-defined. By Theorem \ref{thm:onto} the map $\mathcal E$ is onto.
Suppose that there are $\e,\e'\in \Z/4\Z$ and nonzero scalars $k_i,k_i'\in \F$ for all $i=0,1,2,3$ with $k_0^2=k_0'^{2}=q^{-d-1}$ such that the $\H_q$-modules $E(k_0,k_1,k_2,k_3)^\e$ and $E(k_0',k_1',k_2',k_3')^{\e'}$ are isomorphic. 
Since $q$ is not a root of unity it follows from Lemma \ref{lem:injective} that $\e=\e'$.
By Proposition \ref{prop:E}(ii) it follows that $k_0'=k_0$ and $k_i'=k_i^{\pm 1}$ for all $i=1,2,3$. This shows the injectivity of $\mathcal E$. The result follows.
\hfill $\square$

\section{Proof of Theorem \ref{thm:iso_O}}\label{section:iso_O}

Throughout this section we use the following conventions: 
Let $d\geq 0$ denote an even integer. Assume that $k_0 k_1 k_2 k_3=q^{-d-1}$.
Let $\{v_i\}_{i=0}^d$ denote the $\F$-basis for the $\H_q$-module $O(k_0,k_1,k_2,k_3)$ from Proposition \ref{prop:O}(i).
Define 
$$
N(k_0,k_1,k_2,k_3)
$$
to be the $\H_q$-submodule of $M(k_0,k_1,k_2,k_3)$ generated by $m_{d+1}$. Recall the parameters $\{\varrho_i\}_{i\in \Z}$ and $\{\psi_i\}_{i\in \Z}$ from (\ref{varphi}) and (\ref{varrho}). Under the assumptions we have 
\begin{align}
\varrho_i
&=\left\{
\begin{array}{ll}
(1-q^i) (1-k_0^2 q^i)
\qquad 
&\hbox{if $i$ is even},
\\
(q^{i-d-1}-1)(k_2^{-2} q^{i-d-1}-1)
\qquad 
&\hbox{if $i$ is odd};
\end{array}
\right.
\label{varphi_even}
\\
\psi_i
&=\left\{
\begin{array}{ll}
(1-q^i) (1-k_1^2 q^i)
\qquad 
&\hbox{if $i$ is even},
\\
(q^{i-d-1}-1)(k_3^{-2} q^{i-d-1}-1)
\qquad 
&\hbox{if $i$ is odd}
\end{array}
\right.
\label{varrho_even}
\end{align}
for all $i\in \Z$.

\begin{lem}\label{lem:XYinO}
\begin{enumerate}
\item The action of $X$ on $O(k_0,k_1,k_2,k_3)$ is as follows:
\begin{align*}
(1-k_0 k_3 q^{2\lceil \frac{i}{2}\rceil}X^{(-1)^{i-1}}) v_i
=
\left\{
\begin{array}{ll}
0
\qquad 
&\hbox{if $i=0$},
\\
\varrho_i
v_{i-1}
\qquad 
&\hbox{if $i=1,2,\ldots,d$}.
\end{array}
\right.
\end{align*}

\item The action of $Y$ on $O(k_0,k_1,k_2,k_3)$ is as follows:
\begin{align*}
(1-k_0 k_1 q^{2\lceil \frac{i}{2}\rceil} Y^{(-1)^{i-1}}) v_i
=
\left\{
\begin{array}{ll}
0
\qquad 
&\hbox{if $i=d$},
\\
v_{i+1}
\qquad 
&\hbox{if $i=0,1,\ldots,d-1$}.
\end{array}
\right.
\end{align*}
\end{enumerate}
\end{lem}
\begin{proof}
Similar to Lemma \ref{lem:XYinE}.
\end{proof}

\begin{lem}\label{lem:N_O}
$\{m_i\}_{i=d+1}^\infty$ is an $\F$-basis for $N(k_0,k_1,k_2,k_3)$.
\end{lem}
\begin{proof}
Similar to Lemma \ref{lem:N_E}.
\end{proof}

\begin{lem}\label{lem:O}
There exists a unique $\H_q$-module isomorphism 
$$
M(k_0,k_1,k_2,k_3)/N(k_0,k_1,k_2,k_3)\to O(k_0,k_1,k_2,k_3)
$$ 
that sends $m_i+N(k_0,k_1,k_2,k_3)$ to $v_i$ for all $i=0,1,\ldots,d$.
\end{lem}
\begin{proof}
Similar to Lemma \ref{lem:E}.
\end{proof}

\begin{prop}\label{prop:universal_O}
Let $V$ denote an $\H_q$-module.
If the element
\begin{gather*}
\prod_{i=0}^{d}
(1-k_0 k_1q^{2\lceil \frac{i}{2}\rceil}Y^{(-1)^{i-1}})
\end{gather*}
vanishes at some vector $v$ of $V$ and there is an $\H_q$-module homomorphism $M(k_0,k_1,k_2,k_3)\to V$ that sends $m_0$ to $v$, then there exists an $\H_q$-module homomorphism $O(k_0,k_1,k_2,k_3)\to V$ that sends $v_0$ to $v$.
\end{prop}
\begin{proof}
Similar to Proposition \ref{prop:universal_E}.
\end{proof}

\begin{lem}\label{lem:irr_O}
If the $\H_q$-module $O(k_0,k_1,k_2,k_3)$ is irreducible then the following conditions hold:
\begin{enumerate}
\item $q^i\not=1$ for all $i=2,4,\ldots,d$.

\item $k_0^2,k_2^2\not=q^{-i}$ for all $i=2,4,\ldots,d$.
\end{enumerate}
\end{lem}
\begin{proof}
Similar to Lemma \ref{lem:irr_E}.
\end{proof}

\begin{prop}\label{prop:iso2_O}
Let $\{u_i\}_{i=0}^d$ denote the $\F$-basis for $O(k_1,k_2,k_3,k_0)$ from Proposition {\rm \ref{prop:O}(i)}. 
Then there exists a unique $\H_q$-module homomorphism 
\begin{gather}\label{iso2_O}
O(k_0,k_1,k_2,k_3)\to O(k_1,k_2,k_3,k_0)^{3\bmod{4}}
\end{gather} 
that sends $v_0$ to $u_d$. Moreover {\rm (\ref{iso2_O})} is an isomorphism if and only if the following hold:
\begin{enumerate}
\item $q^i\not=1$ for all $i=2,4,\ldots,d$.

\item $k_1^2,k_3^2\not=q^{-i}$ for all $i=2,4,\ldots,d$.
\end{enumerate}
\end{prop}
\begin{proof}
Let $\e=3\pmod{4}$. 
Note that $(t_0^\e,t_1^\e,t_2^\e,t_3^\e, Y^\e)=(t_3,t_0,t_1,t_2,X)$ by Table \ref{Z/4Z-action}. 
By Proposition \ref{prop:O}(i) the vector $u_d$ satisfies 
\begin{gather*}
t_0 u_d=k_0 u_d,
\qquad 
t_3 u_d= k_3 u_d,
\\
c_1 u_d=(k_1+k_1^{-1}) u_d,
\qquad 
c_2 u_d=(k_2+k_2^{-1}) u_d
\end{gather*}
in $O(k_1,k_2,k_3,k_0)^\e$. By Theorem \ref{thm:universal} there exists a unique $\H_q$-module homomorphism 
$$
M(k_0,k_1,k_2,k_3)\to O(k_1,k_2,k_3,k_0)^\e
$$ that sends $m_0$ to $u_d$. By Lemma \ref{lem:XYinO}(i) we have 
\begin{align}\label{YinO_iso2}
(1-k_0 k_1 q^{2\lceil \frac{i}{2}\rceil}Y^{(-1)^{i-1}}) u_i
=
\left\{
\begin{array}{ll}
0
\qquad 
&\hbox{if $i=0$},
\\
\psi_i
u_{i-1}
\qquad 
&\hbox{if $i=1,2,\ldots,d$}.
\end{array}
\right.
\end{align}
in $O(k_1,k_2,k_3,k_0)^\e$. In particular 
$$
\prod_{i=0}^d
(1-k_0 k_1 q^{2\lceil \frac{i}{2}\rceil}Y^{(-1)^{i-1}}) u_d=0.
$$
Therefore the $\H_q$-module homomorphism (\ref{iso2_O}) exists by Proposition \ref{prop:universal_O}.

Define $w_0=v_0$ and 
$$
w_i=(1-k_0 k_1 q^{d-2\lfloor \frac{i-1}{2}\rfloor}Y^{(-1)^i}) w_{i-1}
\qquad 
\hbox{for all $i=1,2,\ldots,d$}.
$$
By Lemma \ref{lem:XYinO}(ii) a routine induction shows that $w_i$ is equal to 
$$
q^{\lceil\frac{i}{2} \rceil
(d-2\lfloor \frac{i}{2}\rfloor)} (v_i-v_{i-1})
+  
q^{\lfloor\frac{i}{2} \rfloor (d-2\lceil\frac{i}{2} \rceil+2)}
v_{i-1}
$$
plus an $\F$-linear combination of $v_0,v_1,\ldots,v_{i-2}$ for all $i=1,2,\ldots,d$. Hence $\{w_i\}_{i=0}^d$ is an $\F$-basis for $O(k_0,k_1,k_2,k_3)$. Comparing with (\ref{YinO_iso2}) yields that (\ref{iso2_O}) sends 
\begin{eqnarray*}
w_i &\mapsto& 
\left(
\prod_{h=1}^i \psi_{d-h+1}
\right) u_{d-i} 
\qquad 
\hbox{for all $i=0,1,\ldots,d$}.
\end{eqnarray*}
Hence (\ref{iso2_O}) is an isomorphism if and only if $\psi_i\not=0$ for all $i=1,2,\ldots,d$. The latter is equivalent to the conditions (i) and (ii) by (\ref{varrho_even}). The proposition follows.
\end{proof}

\begin{thm}\label{thm:irr_O}
The $\H_q$-module $O(k_0,k_1,k_2,k_3)$ is irreducible if and only if the following conditions hold:
\begin{enumerate}
\item $q^i\not=1$ for all $i=2,4,\ldots,d$.

\item $k_0^2,k_1^2,k_2^2,k_3^2\not=q^{-i}$ for all $i=2,4,\ldots,d$.
\end{enumerate}
\end{thm}
\begin{proof}
($\Rightarrow$): By the irreducibility of $O(k_0,k_1,k_2,k_3)$ the $\H_q$-module homomorphism (\ref{iso2_O}) is an isomorphism. By Lemma \ref{lem:irr_O} and Proposition \ref{prop:iso2_O} the conditions (i) and (ii) follow.

($\Leftarrow$): Consider the operators
\begin{align*}
R &= \prod_{h=1}^{d}
(1-k_0^{-1} k_3^{-1} q^{-2\lceil \frac{h}{2}\rceil} X^{(-1)^{h}}),
\\
S_i &= 
\prod_{h=1}^{d-i}
(1-k_2^{-1} k_3^{-1} q^{1-2\lceil \frac{h}{2}\rceil} Y^{(-1)^{h}})
\qquad 
\hbox{for all $i=0,1\ldots, d$}.
\end{align*}
Using Lemma \ref{lem:XYinO}(i) yields that $R v$ is a scalar multiple of $v_0$ for all $v\in O(k_0,k_1,k_2,k_3)$. In particular, for any $i,j\in \{0,1,\ldots,d\}$ there exists a unique $L_{ij}\in \F$ such that 
\begin{gather*}\label{Lij_O}
R S_i v_j= L_{ij} v_0.
\end{gather*}

Using  Lemma \ref{lem:XYinO}(ii) yields that 
\begin{align}
L_{ij} &=0 \qquad 
\hbox{for all $0\leq i<j\leq d$};
\label{L:lower_O}\\
L_{ij} &=
\left\{
\begin{array}{ll}
(1-k_0^2 k_1^2 q^{i+j}) L_{i,j-1}
+ L_{i-1,j}
-L_{i-1,j-1}
&\hbox{if $i=j \bmod{2}$},
\\
q^{j-i-1}
(L_{i-1,j-1}-L_{i,j-1})+L_{i,j-1}
&\hbox{if $i\not=j \bmod{2}$}
\end{array}
\right.
\label{L:recurrence_O}
\end{align}
for all $i,j\in\{1,2,\ldots,d\}$. 
By (i) and (ii) it follows from Proposition \ref{prop:iso2_O} that (\ref{iso2_O}) is an isomorphism. From this we obtain that 
\begin{gather}\label{Li0:O}
L_{i0}=
\prod_{h=0}^{\lfloor \frac{i-1}{2} \rfloor}
(1-q^{2h-d})
\prod_{h=1}^{\lfloor \frac{i}{2}\rfloor}
(1-k_1^2 k_2^2 q^{d+2h})
\prod_{h=1}^{d-i}\psi_{d-h+1}
\qquad 
\hbox{for all $i=0,1,\ldots,d$}.
\end{gather}
Solving the recurrence relation (\ref{L:recurrence_O}) with the initial conditions (\ref{L:lower_O}) and (\ref{Li0:O}) yields  that $L_{ij}$ is equal to 
\begin{align*}
&
(-q^{d+i-\lceil \frac{j}{2}\rceil} k_1^2 k_2^2)^{\lceil \frac{j}{2}\rceil}
\prod_{h=1}^{\lfloor \frac{j}{2} \rfloor}(1-q^{2h})^{-1}
\prod_{h=\lceil \frac{j}{2}\rceil}^{\lfloor \frac{i-1}{2} \rfloor}
(1-q^{2h-d})
\prod_{h=1}^{\lfloor \frac{i}{2} \rfloor-\lceil \frac{j}{2}\rceil}
(1-k_1^2 k_2^2 q^{d+2h})
\prod_{h=1}^j \varrho_h
\prod_{h=1}^{d-i} \psi_{d-h+1}
\\
&
\qquad 
\times 
\left\{
\begin{array}{ll}
\displaystyle
(k_3^{-2} q^{j-d-1}-k_0^{2}
q^{j+1}-q^{j-i}+1)
\prod_{h=1}^{\frac{j-1}{2}}
(1-q^{2h-i-1})
\qquad 
&\hbox{if $i,j$ are odd and $i>j$},
\\
\displaystyle
-k_0^{2}
q^{i+1}
\prod_{h=1}^{\frac{i-1}{2}}
(1-q^{2h-i-1})
\qquad 
&\hbox{if $i,j$ are odd and $i=j$},
\\
\displaystyle
\prod_{h=1}^{\frac{j}{2}}
(1-q^{2h-i-1})
\qquad 
&\hbox{if $i$ is odd and $j$ is even},
\\
\displaystyle
\prod_{h=1}^{\lceil \frac{j}{2}\rceil}(q-q^{2h-i-1})
\qquad 
&\hbox{if $i$ is even}
\end{array}
\right.
\end{align*}
for all $0\leq j\leq i\leq d$. In particular $L_{ii}$ is equal to 
$$
(-q^{d+\lfloor \frac{i}{2}\rfloor} k_1^2 k_2^2)^{\lceil \frac{i}{2}\rceil}
\prod_{h=1}^{\lfloor \frac{i}{2} \rfloor}(1-q^{2h})^{-1}
\prod_{h=1}^i \varrho_h
\prod_{h=1}^{d-i} \psi_{d-h+1}
\times 
\left\{
\begin{array}{ll}
\displaystyle
-k_0^{2}
q^{i+1}
\prod_{h=1}^{\frac{i-1}{2}}
(1-q^{2h-i-1})
\qquad 
&\hbox{if $i$ is odd},
\\
\displaystyle
\prod_{h=1}^{\frac{i}{2}}(q-q^{2h-i-1})
\qquad 
&\hbox{if $i$ is even}.
\end{array}
\right.
$$
It follows from (i) and (ii) that $L_{ii}$ is nonzero for all $i=0,1,\ldots,d$.
To see the irreducibility of $O(k_0,k_1,k_2,k_3)$ it remains to apply a routine argument similar to the proof of Theorem \ref{thm:irr}.
\end{proof}

We now give our proof for Theorem \ref{thm:iso_O}.

\medskip

\noindent{\it Proof of Theorem \ref{thm:iso_O}.}
(i): Immediate from Proposition \ref{prop:iso2_O} and Theorem \ref{thm:irr_O}.

(ii): By Theorem \ref{thm:irr_O} the $\H_q$-module $O(k_1,k_2,k_3,k_0)$ is irreducible. It follows from (i) that $O(k_1,k_2,k_3,k_0)$ is isomorphic to the $\H_q$-module $O(k_2,k_3,k_0,k_1)^{3\bmod{4}}$. Combined with (i) we obtain (ii).

(iii): By Theorem \ref{thm:irr_O} the $\H_q$-module $O(k_2,k_3,k_0,k_1)$ is irreducible.It follows from (i) that $O(k_2,k_3,k_0,k_1)$ is isomorphic to the $\H_q$-module $O(k_3,k_0,k_1,k_2)^{3\bmod{4}}$. Combined with (ii) we obtain (iii).
\hfill $\square$

\section{Proof of Theorem \ref{thm:odd}}\label{section:odd}

In the final section we are devoted to proving Theorem \ref{thm:odd}.

\begin{thm}\label{thm:onto_O}
Assume that $\F$ is algebraically closed and $q$ is not a root of unity. Let $d\geq 0$ denote an even integer. If $V$ is a $(d+1)$-dimensional irreducible $\H_q$-module, then there exist nonzero $k_0,k_1,k_2,k_3\in \F$ with $k_0 k_1 k_2 k_3=q^{-d-1}$ such that the $\H_q$-module $O(k_0,k_1,k_2,k_3)$ is isomorphic to $V$.
\end{thm}
\begin{proof}
Suppose that $V$ is a $(d+1)$-dimensional irreducible $\H_q$-module. 
According to Propositions \ref{prop:X_eigenvector} and \ref{prop:Y_eigenvector} we may divide the argument into the following cases (i)--(iv):
\begin{enumerate}
\item[(i)] $t_0,t_3$ have a simultaneous eigenvector and $t_0,t_1$ have a simultaneous eigenvector in $V$.

\item[(ii)] $t_1,t_2$ have a simultaneous eigenvector and $t_0,t_1$ have a simultaneous eigenvector in $V$.

\item[(iii)] $t_1,t_2$ have a simultaneous eigenvector and $t_2,t_3$ have a simultaneous eigenvector in $V$.

\item[(iv)] $t_0,t_3$ have a simultaneous eigenvector and $t_2,t_3$ have a simultaneous eigenvector in $V$.
\end{enumerate}

(i): Let $v$ denote a simultaneous eigenvector of $t_0$ and $t_3$ in $V$. Choose $k_0$ and $k_3$ as the eigenvalues $t_0$ and $t_3$ corresponding to $v$, respectively. Let $w$ denote a simultaneous eigenvector of $t_0$ and $t_1$ in $V$. Choose $k_1$ as the eigenvalue of $t_1$ corresponding to $w$. By Lemma \ref{lem:Schur} the element $c_1$ acts on $V$ as scalar multiplication by $k_1+k_1^{-1}$ and there exists a nonzero scalar $k_2\in \F$ such that $c_2$ acts on $V$ as scalar multiplication by $k_2+k_2^{-1}$. Note that the eigenvalues of $t_2$ in $V$ are $k_2$ or $k_2^{-1}$. If  both $k_2$ and $k_2^{-1}$ are the eigenvalues of $t_2$ in $V$, then we require $k_2$ to satisfy that the algebraic multiplicity of $k_2$ is greater than or equal to the algebraic multiplicity of $k_2^{-1}$. 
By Theorem \ref{thm:universal} there exists a unique $\H_q$-module homomorphism 
\begin{gather}\label{M->V_O}
M(k_0,k_1,k_2,k_3) \to V
\end{gather}
that sends $m_0$ to $v$.

Let $v_i$ denote the image of $m_i$ under (\ref{M->V_O}) for all $i=0,1,2,\ldots$. Suppose that there is an $i\in \{1,2,\ldots, d\}$ such that $v_i$ is an $\F$-linear combination of $v_0,v_1,\ldots,v_{i-1}$. Let $W$ denote the $\F$-subspace of $V$ spanned by $v_0,v_1,\ldots,v_{i-1}$. By Proposition \ref{prop:P}, $W$ is invariant under $\{t_i^{\pm 1}\}_{i=0}^3$. Hence $W$ is a proper $\H_q$-submodule of $V$, a contradiction to the irreducibility of $V$. Therefore 
$
\{v_i\}_{i=0}^d
$
is an $\F$-basis for $V$.

For $u\in V$ let $[u]$ denote the coordinate vector of $u$ relative to $\{v_i\}_{i=0}^d$. Let $a_i$ and $b_i$ denote the $i$-entries of $[v_{d+1}]$ and $[w]$ for all $i=0,1,\ldots,d$, respectively. Since $w$ is a $k_1$-eigenvector of $t_1$ we have 
\begin{gather}\label{t1w=k1w}
[t_1 w]=k_1[w].
\end{gather}

We now show that $a_0=0$. To do this, we divide the argument into two cases: (a) $w$ is a $k_0$-eigenvector of $t_0$; (b) $w$ is not a $k_0$-eigenvector of $t_0$. 

(a): Using Proposition \ref{prop:P}(i) yields that the matrix representing $t_0$ with respect to the basis 
\begin{gather*}
v_0,
\quad 
v_{i-1}+q^{-i} (v_i-v_{i-1})
\quad \hbox{for $i=2,4,\ldots,d$},
\quad 
v_i 
\quad \hbox{for $i=1,3,\ldots,d-1$}
\end{gather*}
for $V$ is 
\begin{gather*}
\begin{pmatrix}
k_0   &\rvline &{\bf 0}  &\rvline &{\bf 0} 
\\
\hline
{\bf 0}  &\rvline
&k_0 I_{\frac{d}{2}}
 &\rvline & -k_0^{-1} I_{\frac{d}{2}}
\\
\hline
{\bf 0} &  \rvline 
&{\bf 0} &  \rvline
&k_0^{-1} I_{\frac{d}{2}} 
\end{pmatrix}.
\end{gather*}
By the rank-nullity theorem the $k_0$-eigenspace of $t_0$ in $V$ has dimension $\frac{d}{2}+1$. Hence the $k_0$-eigenspace of $t_0$ in $V$ is spanned by 
\begin{gather}\label{k0basis}
v_0, 
\qquad 
v_{i-1}+q^{-i} (v_i-v_{i-1})
\quad 
\hbox{for all $i=2,4,\ldots,d$}.
\end{gather}
Hence $w$ is an $\F$-linear combination of (\ref{k0basis}). Since $w\not=0$ there exists an $i\in\{0,2,\ldots,d\}$ with $b_i\not=0$. 
Let $k$ denote the largest even integer such that $b_k\not=0$. 
Suppose that $k\not=d$. Then $b_{k+2}=0$. By (\ref{k0basis}) the coefficient $b_{k+1}=0$. By Proposition \ref{prop:P}(i) the $(k+1)$-entry of $[t_1 w]$ is equal to $k_1^{-1} b_k\not=0$, a contradiction to (\ref{t1w=k1w}). Hence $b_d\not=0$. Evaluating the first entry of $[t_1 w]$ by using Proposition \ref{prop:P}(i), we obtain from (\ref{t1w=k1w}) that
$$
k_1^{-1} a_0 b_d+ k_1 b_0=k_1 b_0.
$$
Since $b_d\not=0$ it follows that $a_0=0$.

(b):  Observe that $w$ is a $k_0^{-1}$-eigenvector of $t_0$ and $k_0^2\not=1$. Using Proposition \ref{prop:P}(i) yields that the $k_0^{-1}$-eigenspace of $t_0$ is spanned by 
\begin{gather}\label{k0inv-basis}
v_i-(1-k_0^2 q^i) v_{i-1} 
\quad 
\hbox{for all $i=2,4,\ldots,d$}
\end{gather}
Hence $w$ is an $\F$-linear combination of (\ref{k0inv-basis}). Evaluating the first entry of $[t_1 w]$ by using Proposition \ref{prop:P}(i), we obtain from (\ref{t1w=k1w}) that
$$
k_1^{-1} a_0 b_d=0.
$$
By a similar argument to the case (a), we see that $b_d\not=0$. Hence $a_0=0$.

We now show that $a_i=0$ for all $i=1,2,\ldots,d$. Let $\{\varrho_i\}_{i\in \Z}$ denote the parameters (\ref{varphi_even}) corresponding to the current parameters $k_0,k_1,k_2,k_3$. Using Lemma \ref{lem:XYinP}(i) yields that 
$(1-k_0 k_3 q^{d+2} X) v_i$ is equal to 
\begin{equation}
\begin{split}\label{XV}
\left\{
\begin{array}{ll}
(1-k_0^2 k_3^2 q^{d+2}) v_0
\qquad 
&\hbox{if $i=0$},
\\
(1-q^{d-i+1}) v_i
+q^{d-i+1} \varrho_i v_{i-1}
\qquad 
&\hbox{if $i=1,3,\ldots,d-1$},
\\
(1-k_0^2 k_3^2 q^{d+i+2}) v_i
-
q^{d-i+2}\varrho_i
(v_{i-1}-\varrho_{i-1} v_{i-2})
\qquad 
&\hbox{if $i=2,4,\ldots,d$}.
\end{array}
\right.
\end{split}
\end{equation}
By Lemma \ref{lem:XYinP}(i) with $i=d+1$ we have
\begin{gather}\label{Xud+1}
[(1-k_0 k_3 q^{d+2} X) v_{d+1}]=
\begin{pmatrix}
0
\\
0
\\
\vdots
\\
0
\\
\varrho_{d+1}
\end{pmatrix}.
\end{gather}
Evaluating the first $d$ entries of the left-hand side of (\ref{Xud+1}) by using (\ref{XV}), we obtain that 
\begin{gather}\label{ak_O}
\varrho_i a_i
=\left\{
\begin{array}{ll}
(k_0^2 k_3^2 q^{d+i+1}-1) a_{i-1}
\qquad 
&\hbox{for $i=1,3,\ldots,d-1$},
\\
(q^{i-d-2}-1) a_{i-1}
\qquad 
&\hbox{for $i=2,4,\ldots,d$}.
\end{array}
\right.
\end{gather} 
Suppose on the contrary that there exists an $i\in\{1,2,\ldots,d\}$ with $a_i \not=0$ and $a_{i-1}=0$. 
By (\ref{ai}) the scalar $\varrho_i=0$. 
Let $W$ denote the $\F$-subspace of $V$ spanned by $v_i,v_{i+1},\ldots,v_d$. Combined with Proposition \ref{prop:P} yields that $W$ is invariant under $\{t_i^{\pm 1}\}_{i=0}^3$. Hence $W$ is a proper $\H_q$-submodule of $V$, a contradiction to the irreducibility of $V$. Therefore $a_i=0$ for all $i=0,1,\ldots,d$. In other words 
\begin{gather}\label{ud+1=0_O}
v_{d+1}=0.
\end{gather}

Combined with Proposition \ref{prop:P}(i) yields that the roots of the characteristic polynomial of $t_2$ in $V$ are 
$$
\underbrace{k_2, k_2, \ldots, k_2}_{\hbox{{\tiny $d/2$ copies}}},
\underbrace{k_2^{-1},k_2^{-1},\ldots, k_2^{-1}}_{\hbox{{\tiny $d/2$ copies}}},
k_0^{-1} k_1^{-1} k_3^{-1} q^{-d-1}.
$$
By the choice of $k_2$ we have 
$$
k_0 k_1 k_2 k_3=q^{-d-1}.
$$
By Lemma \ref{lem:XYinP}(ii) we have 
\begin{gather}\label{e:Yv_O}
\prod_{i=0}^d(1-k_0 k_1 q^{2\lceil \frac{i}{2}\rceil} Y^{(-1)^{i-1}}) v_0=v_{d+1}.
\end{gather}
By (\ref{ud+1=0_O}) the right-hand side of (\ref{e:Yv_O}) is zero. Thus, it follows from  Proposition \ref{prop:universal_O} that there exists a nontrivial $\H_q$-module homomorphism 
\begin{gather}\label{O->V}
O(k_0,k_1,k_2,k_3)\to V.
\end{gather}
Since the $\H_q$-module $V$ is irreducible it follows that (\ref{O->V}) is onto. Since $O(k_0,k_1,k_2,k_3)$ and $V$ are of dimension $d+1$ it follows that (\ref{O->V}) is an isomorphism. Therefore the theorem holds for the case (i).

(ii): Let $\e=1\pmod 4$. Since $t_0^\e=t_1,t_1^\e=t_2,t_3^\e=t_0$ by Table \ref{Z/4Z-action} the elements $t_0,t_3$ have a simultaneous eigenvector and $t_0,t_1$ have a simultaneous eigenvector in $V^\e$. By (i) there are nonzero $k_0,k_1,k_2,k_3\in \F$ with $k_0 k_1 k_2 k_3 =q^{-d-1}$ such that $O(k_0,k_1,k_2,k_3)$ is isomorphic to $V^\e$. Hence  $O(k_0,k_1,k_2,k_3)^{-\e}$ is isomorphic to $V$. Combined with Theorem \ref{thm:iso_O}(i) the theorem holds for the case (ii).

(iii): Let $\e=2\pmod 4$. Since $t_0^\e=t_2,t_1^\e=t_3,t_3^\e=t_1$ by Table \ref{Z/4Z-action} the elements $t_0,t_3$ have a simultaneous eigenvector and $t_0,t_1$ have a simultaneous eigenvector in $V^\e$. By (i) there are nonzero $k_0,k_1,k_2,k_3\in \F$ with $k_0 k_1 k_2 k_3 =q^{-d-1}$ such that $O(k_0,k_1,k_2,k_3)$ is isomorphic to $V^\e$. Hence  $O(k_0,k_1,k_2,k_3)^{-\e}$ is isomorphic to $V$. Combined with Theorem \ref{thm:iso_O}(ii) the theorem holds for the case (iii).

(iv): Let $\e=3\pmod 4$. Since $t_0^\e=t_3,t_1^\e=t_0,t_3^\e=t_2$ by Table \ref{Z/4Z-action} the elements $t_0,t_3$ have a simultaneous eigenvector and $t_0,t_1$ have a simultaneous eigenvector in $V^\e$. By (i) there are nonzero $k_0,k_1,k_2,k_3\in \F$ with $k_0 k_1 k_2 k_3 =q^{-d-1}$ such that $O(k_0,k_1,k_2,k_3)$ is isomorphic to $V^\e$. Hence  $O(k_0,k_1,k_2,k_3)^{-\e}$ is isomorphic to $V$. Combined with Theorem \ref{thm:iso_O}(iii) the theorem holds for the case (iv).
\end{proof}

\begin{lem}\label{lem:injective_O}
Let $d\geq 0$ denote an even integer. 
For any nonzero $k_0,k_1,k_2,k_3\in \F$ with $k_0 k_1 k_2 k_3=q^{-d-1}$ the determinants of $t_0,t_1,t_2,t_3$ on $O(k_0,k_1,k_2,k_3)$ are $k_0,k_1,k_2,k_3$ respectively.
\end{lem}
\begin{proof}
It is routine to verify the lemma by using Proposition \ref{prop:O}(i).
\end{proof}

We finish this paper with the proof for Theorem \ref{thm:odd}.

\medskip

\noindent {\it Proof of Theorem \ref{thm:odd}.}
By Theorem \ref{thm:irr_O} the map $\mathcal O$ is well-defined. By Theorem \ref{thm:onto_O} the map $\mathcal O$ is onto. Since each element of $\H_q$ has the same determinant on all isomorphic finite-dimensional $\H_q$-modules, it follows from Lemma \ref{lem:injective_O} that $\mathcal O$ is one-to-one.
\hfill $\square$

\subsection*{Acknowledgements}

The research is supported by the Ministry of Science and Technology of Taiwan under the project MOST 106-2628-M-008-001-MY4.

\bibliographystyle{amsplain}
\bibliography{MP}

\end{document}